\documentclass[11pt,reqno]{amsart}
\usepackage[utf8]{inputenc}
\usepackage{enumitem}
\usepackage{amssymb}
\usepackage{mathtools}
\usepackage{pdfsync}
\usepackage[english]{babel}
\usepackage[pdftex,pagebackref,colorlinks=true,urlcolor=blue,linkcolor=blue,citecolor=red]{hyperref}
\usepackage{fullpage}
\usepackage{color}
\usepackage{amsmath}
\usepackage{amsfonts}
\usepackage{mathrsfs}
\usepackage{t1enc, graphicx}
\usepackage{verbatim}
\usepackage{bbm}
\usepackage[colorinlistoftodos]{todonotes}
\usepackage{bm}
\usepackage{upgreek}
\usepackage[mathcal]{eucal}
\usepackage{tikz-cd}
\usepackage{adjustbox}
\usepackage{marginnote}
\usepackage[nocompress,noadjust]{cite}
\usepackage{dsfont}
\usepackage{aliascnt}
\usepackage[capitalize]{cleveref}

\newcommand{\C}{\mathbb{C}}

\newcommand{\Q}{\mathbb{Q}}

\newcommand{\Z}{\mathbb{Z}}
\newcommand{\N}{\mathbb{N}}

\newcommand{\RP}{\mathbf{RP}}

\newcommand{\Krat}{\mathbf{K}_{\operatorname{rat}}}

\newcommand{\bQ}{\bm{Q}}

\newcommand{\bx}{{\boldsymbol x}}
\newcommand{\by}{{\boldsymbol y}}

\newcommand{\bz}{{\boldsymbol z}}

\newcommand{\cB}{{\mathcal B}}
\newcommand{\cF}{{\mathcal F}}
\newcommand{\cC}{{\mathcal C}}
\newcommand{\cP}{{\mathcal P}}

\newcommand{\cO}{{\mathcal O}}

\DeclareMathOperator{\vdc}{-vdC}

\newtheorem{theorem}{Theorem}[section]
\crefname{theorem}{Theorem}{Theorems}

\newaliascnt{proposition}{theorem}
\newtheorem{proposition}[proposition]{Proposition}
\aliascntresetthe{proposition}
\crefname{proposition}{Proposition}{Propositions}

\newaliascnt{conjecture}{theorem}
\newtheorem{conjecture}[conjecture]{Conjecture}
\aliascntresetthe{conjecture}
\crefname{conjecture}{Conjecture}{Conjectures}
\newtheorem*{conjecture*}{Conjecture}

\newaliascnt{problem}{theorem}

\aliascntresetthe{problem}
\crefname{problem}{Problem}{Problems}

\newaliascnt{lemma}{theorem}
\newtheorem{lemma}[lemma]{Lemma}
\aliascntresetthe{lemma}
\crefname{lemma}{Lemma}{Lemmas}

\newaliascnt{claim}{theorem}
\newtheorem{claim}[claim]{Claim}
\aliascntresetthe{claim}
\crefname{claim}{Claim}{Claims}

\newaliascnt{corollary}{theorem}
\newtheorem{corollary}[corollary]{Corollary}
\aliascntresetthe{corollary}
\crefname{corollary}{Corollary}{Corollaries}
\newtheorem*{corollary*}{Corollary}

\newtheorem*{theorem*}{Theorem}

\theoremstyle{definition}

\newaliascnt{definition}{theorem}
\newtheorem{definition}[definition]{Definition}
\aliascntresetthe{definition}
\crefname{definition}{Definition}{Definitions}
\newtheorem*{definition*}{Definition}

\newaliascnt{question}{theorem}

\aliascntresetthe{question}
\crefname{question}{Question}{Questions}

\newtheorem{maintheorem}{Theorem}

\crefname{maintheorem}{Theorem}{Theorems}

\theoremstyle{remark}

\newaliascnt{example}{theorem}
\newtheorem{example}[example]{Example}
\aliascntresetthe{example}
\crefname{example}{Example}{Examples}

\newaliascnt{remark}{theorem}

\aliascntresetthe{remark}
\crefname{remark}{Remark}{Remarks}

\newaliascnt{exercise}{theorem}

\aliascntresetthe{exercise}
\crefname{exercise}{Exercise}{Exercises}
\title{A saturation theorem for distinct-degree polynomials and an application to joint transitivity}

\date{}

 \author{Axel \'Alvarez}
 \address[Axel \'Alvarez]{Departamento de Ingenier\'{\i}a Matem\'atica and Centro de Modelamiento Matem{\'a}tico, Universidad de Chile \& IRL 2807 - CNRS, Beauchef 851, Santiago, Chile} \email{aalvarez@dim.uchile.cl}

\begin{document}



\begin{abstract}
    We prove, modulo almost one-to-one extensions, that the topological rational Kronecker factors of a minimal $\Z^k$-system are the topological characteristic factors for commuting transformations along polynomials with distinct degrees. As an application, we obtain necessary and sufficient conditions for the joint transitivity of these polynomial iterates.
\end{abstract}

\maketitle

\section{Introduction}

\subsection{The joint ergodicity problem}

The study of multiple ergodic averages was initiated in the seminal work of Furstenberg \cite{Furstenberg_ergodic_szemeredi:1977}, where an ergodic-theoretic proof of Szemerédi’s theorem on arithmetic progressions was given. Since then, many variants of multiple ergodic averages have been studied. A fundamental class is given by polynomial iterates for commuting transformations, namely averages of the form\begin{align*}
    \dfrac{1}{N} \sum_{n=1}^{N} T_1^{p_1(n)}f_1\cdot \ldots \cdot T_k^{p_k(n)}f_k
\end{align*}

where $f_1,\dots,f_k\in L^{\infty}$, $p_1,\dots,p_k\in \Z[t]$ and $T_1,\dots,T_k$ are commuting invertible measure preserving transformations of a probability space $(X,\mu)$. 

The norm convergence of such averages was proved in generality by Walsh \cite{Walsh12}, who established the convergence in $L^2$ for multiple polynomial averages associated with nilpotent groups of measure preserving transformations. Once convergence is known, a natural question is to determine when the limit is the expected one, that is, the product of the integrals of the functions. This is the joint ergodicity problem. More generally, one may ask the same question for integer-valued sequences of iterates.

\begin{definition}
    Let $(X,\mu,T_1,\dots,T_k)$ be a measure preserving system with commuting and invertible transformations and let $(a_1(n))_n, \dots,$ $ (a_k(n))_n$ be integer-valued sequences. We say that $(T_1^{a_1(n)})_n,\dots,(T_k^{a_k(n)})_n$ are \emph{jointly ergodic} (for $\mu$) if for any $f_1,\dots,f_k\in L^\infty(\mu)$, we have\begin{align*}
        \lim_{N\to\infty}\dfrac{1}{N} \sum_{n=1}^{N} T_1^{a_1(n)}f_1\cdot \ldots \cdot T_k^{a_k(n)}f_k = \int_X f_1\, d\mu \cdot \ldots \cdot \int_X f_k\, d\mu,
    \end{align*} 

    where the limit is taken in $L^2(\mu)$. When $k=1$, we also say that $(T_1^{a_1(n)})_n$ is \emph{ergodic} for $\mu$.
\end{definition}

The first characterization of joint ergodicity for linear iterates is due to Berend and Bergelson \cite{Berend_Bergelson_joint_ergodicity:1984}. In recent years, this problem has attracted considerable attention beyond the linear case, and has been studied extensively for several other classes of sequences (for example, see \cite{Donoso_Ferre_Koutsogiannis_Sun_multicorr_joint_erg:2024,Donoso_Koutsogiannis_Sun_joint_erg_poly_growth:2023,Donoso_Koutsogiannis_Sun_seminorms_polynomials_joint_ergodicity:2022,Donoso_Koutsogiannis_Kuca_Sun_Tsinas_resolving_joint_ergodicity_Hardy:2025,Donoso_Koutsogiannis_Kuca_Sun_Tsinas_seminorm_joint_ergodicity_independent_Hardy:2025,Frantzikinakis_Kuca_joint_erg_comm_poly:2025,Bergelson_Leibman_Son_joint_erg_generalized_linear:2016,Tsinas_joint_erg_Hardy:2023,Frantzikinakis_joint_ergodicity_sequences:2023,Koutsogiannis_Sun_total_joint_ergodicity:2023,Frantzikinakis_joint_erg_primes:2022,Frantzikinakis_Kuca_seminorm_control_pairwise_dependent_pol:2023}).

\subsection{Topological characteristic factors}

In the study of multiple ergodic averages, characteristic factors play a crucial role. The notion of characteristic factors was first introduced in a paper by Furstenberg and Weiss \cite{Furstenberg_Weiss_ergodic_thm_double:1996}, and its relevance was solidified in the groundbreaking work of Host and Kra \cite{Host_Kra_nonconventional_averages_nilmanifolds:2005}. More recently, characteristic factors have also played an important role in the study of joint ergodicity.

A topological counterpart of characteristic factors was first studied by Glasner \cite{Glasner_top_erg_decomposition:1994}.  In \cite{Glasner_Huang_Shao_Weiss_Ye_Topological_characteristic_factors:2020}, the authors significantly advanced Glasner’s earlier work on topological characteristic factors by showing that, up to almost one-to-one extensions, the maximal $\infty$-step pro-nilfactor of a minimal system is a topological characteristic factor along arithmetic progressions. More precisely, for a dense G$_\delta$ set of points, the orbit closure of $\{(T^nx,\dots,T^{dn}x):n\in \Z\}$ contains all points whose projections to the maximal $\infty$-step pro-nilfactor lie in the corresponding orbit closure on the factor.  In \cite{Qiu_poly_orbits_tot_minimal:2023}, Qiu extended this result to distinct nonconstant polynomial iterates. This was later refined by Ye and Yu \cite{Ye_Yu_polynomial_saturation:2025}, who showed that, for a fixed polynomial family, one can replace the maximal $\infty$-step pro-nilfactor by a suitable $k$-step pro-nilfactor, where $k$ depends only on the family of polynomials. The result of \cite{Glasner_Huang_Shao_Weiss_Ye_Topological_characteristic_factors:2020} has inspired extensive research on topological characteristic factors (see for instance \cite{Alvarez_regionally_proximal_tfc_group_actions:2026,Qiu_poly_orbits_tot_minimal:2023,Qiu_Yu_saturated_cubes_measure:2023,Shao_Xu_saturation_R_flows:2025,Qiu_Xu_Ye_Yu_saturation_product:2025,Wu_Yu_saturation_product_pol:2026}).

The results of \cite{Glasner_Huang_Shao_Weiss_Ye_Topological_characteristic_factors:2020,Qiu_poly_orbits_tot_minimal:2023,Ye_Yu_polynomial_saturation:2025} deal with systems generated by a single transformation. The case of several commuting transformations is much less developed in the topological setting. This contrasts with the measurable setting, where characteristic factors for multiple averages involving commuting transformations have been extensively studied. In particular, Chu, Frantzikinakis, and Host \cite{Chu_Frantzikinakis_Host_ergodic_averages_distinct_degree:2011} proved such a result for polynomial averages of distinct degrees. Later, Frantzikinakis and Kuca {\cite[Theorem 2.10]{Frantzikinakis_Kuca_joint_erg_comm_poly:2025}} improved this result by showing that the rational Kronecker factor of each transformation is characteristic. Under transitivity assumptions, the following theorem can be seen as a topological version of the result of Chu, Frantzikinakis, and Host, with the improvement obtained by Frantzikinakis and Kuca.

\begin{maintheorem}\label{thm: TFC_distinct_deg}
    Let $(X,S_1,\dots,S_k)$ be a minimal $\Z^k$-system, let $T_1,\dots,T_d\in \langle S_{1},\dots,S_{k}\rangle$ be transitive transformations, and let $p_1,\dots,p_d$ be nonconstant integer polynomials with distinct degrees such that $p_{i}(0)=0$ for $i=1,\dots, d$. Then there is a dense G$_\delta$ subset $\Omega$ of $X$ such that, for every $x\in \Omega$ and $x_1,\dots,x_d\in X$ satisfying, for each $j=1,\dots,d$, $(x,x_j)\in \Krat(X,T_i)$\footnote{See \cref{subsec: Krat} for the definition.} for some $i\in \{1,\dots,d\}$, one has \begin{align*}
        (x_1,\dots,x_d)\in \overline{\{ (T_1^{p_1(n)}x,\dots,T_d^{p_d(n)}x):n\in\Z \}},
    \end{align*}
\end{maintheorem}

Actually, \cref{thm: TFC_distinct_deg} is a consequence of a stronger theorem (\cref{thm: TFC_nice_pol}), which gives a saturation result for a broader class of families of polynomial tuples satisfying certain degree conditions.

\subsection{The joint transitivity problem}

A central question in topological dynamics is to understand when several orbits can be made dense simultaneously. For a single transformation, this is the usual notion of transitivity. For several transformations, it is called joint transitivity (also called $\Delta$-transitivity in \cite{Huang_Shao_Ye_top_correspondence_multiple_averages:2019}).

\begin{definition}
    Let $(X,S_1,\dots,S_k)$ be a $\Z^k$-system, $T_1,\dots,T_d\in \langle S_1,\dots,S_k\rangle$ and  $(a_1(n))_n, \dots,$ $ (a_d(n))_n$ be integer-valued sequences. We say that $(T_1^{a_1(n)})_n,\dots,(T_d^{a_d(n)})_n$ are \emph{jointly transitive} if there is a dense G$_\delta$ subset $\Omega$ of $X$ such that for all $x\in \Omega$, the set \begin{align*}
        \{(T_1^{a_1(n)}x,\dots,T_d^{a_d(n)}x):n\in \Z\}
    \end{align*}

    is dense in $X^d$. In the case $d=1$, we simply say that $(T_1^{a_1(n)})_n$ is \emph{transitive}. If, moreover, $a_1(n)=n$ for every $n\in\Z$, then we say that $T_1$ is transitive.
\end{definition}

An early result in this direction is due to Glasner \cite{Glasner_top_erg_decomposition:1994}, who showed that, in every minimal weakly mixing $\mathbb Z$-system $(X,T)$, the sequences $(T^{a_1n})_n,\dots,(T^{a_d n})_n$ are jointly transitive whenever $a_1,\dots,a_d$ are distinct nonzero integers. This was later extended to polynomial iterates for nilpotent group actions by Huang, Shao and Ye \cite{Huang_Shao_Ye_top_correspondence_multiple_averages:2019}, and to generalized polynomials for $\mathbb Z$-systems by Zhang and Zhao \cite{Zhang_Zhao_topological_mult_rec_WM_GP:2021}.

The results mentioned above require weak mixing assumptions. For general minimal $\mathbb Z^k$-systems, Donoso, Koutsogiannis and Sun \cite{Donoso_Koutsogiannis_Sun_joint_transitivity:2025} recently obtained a complete characterization of joint transitivity for linear iterates. Their result can be viewed as a topological analogue of the classical joint ergodicity criterion of Berend and Bergelson.

In \cite[Problem 1.8]{Donoso_Koutsogiannis_Sun_joint_transitivity:2025}, the authors ask for a characterization of joint transitivity for polynomial iterates. As a consequence of \cref{thm: TFC_distinct_deg}, we obtain a characterization in the case where the polynomials have distinct degrees. More precisely, we prove the following result.\begin{maintheorem}\label{thm:B}
    Let $(X,S_1,\dots,S_k)$ be a minimal $\Z^k$-system, let $T_1,\dots,T_d\in \langle S_1,\dots,S_k\rangle$ be transitive transformations, and let $p_1,\dots,p_d$ be nonconstant integer polynomials with distinct degrees. Then $(T_1^{p_1(n)})_{n},\dots,(T_d^{p_d(n)})_{n}$ are jointly transitive if and only if $(T_1^{p_1(n)}\times \dots\times T_{d}^{p_{d}(n)})_{n}$ is transitive on $X^d$.
\end{maintheorem}

We expect that the distinct-degree assumption can be replaced by the weaker assumption of pairwise independence. More precisely, we propose the following conjecture, which may be viewed as a topological analogue of {\cite[Theorem 2.16]{Frantzikinakis_Kuca_joint_erg_comm_poly:2025}}.
\begin{conjecture}
    Let $(X,S_1,\dots,S_k)$ be a minimal $\Z^k$-system, let $T_1,\dots,T_d\in \langle S_1,\dots,S_k\rangle$ be transitive transformations, and let $p_1,\dots,p_d$ be pairwise independent integer polynomials. Then $(T_1^{p_1(n)})_{n},\dots,$ $(T_d^{p_d(n)})_{n}$ are jointly transitive if and only if $(T_1^{p_1(n)}\times \dots\times T_{d}^{p_{d}(n)})_{n}$ is transitive on $X^d$.
\end{conjecture}

The proof of \cref{thm:B} uses \cref{thm: TFC_distinct_deg}. Thus, the above conjecture would follow from a version of \cref{thm: TFC_distinct_deg} for pairwise independent polynomials.

\subsection*{Acknowledgments}

The author is grateful to Sebastián Donoso for his guidance during the preparation of this article and for his helpful comments. The author also thanks the University of Science and Technology of China for its hospitality while this work was being completed. The author was supported by ANID-Subdirección de Capital Humano/Doctorado Nacional/2025-21251865 and Centro de Modelamiento Matemático (CMM) FB210005, BASAL funds for centers of excellence from ANID-Chile.

\section{Background}

\subsection{Topological dynamics}

A \emph{$\Z^k$-system} is a tuple $(X,S_1,\dots,S_k)$, where $X$ is a compact metric space and $S_1,\dots,S_k\colon X\to X$ are commuting homeomorphisms; that is, $S_iS_j=S_jS_i$ for all $1\leq i,j \leq k$. We denote by $\langle S_1,\dots,S_k\rangle$ the group generated by $S_1,\dots,S_k$.

For any $\Z^k$-system $(X,S_1,\dots,S_k)$ and $m = (m_1,\dots,m_k)\in\Z^k$, we write $S_m = S_1^{m_1}S_2^{m_2}\cdots S_k^{m_k}$. We say that $(X,S_1,\dots,S_k)$ is \emph{minimal} if, for every $x\in X$, the set $\{S_m x: m\in \Z^k\}$ is dense in $X$. 

%Let $(R_n)_n$ be a sequence of homeomorphisms of $X$. We say that $(R_n)_n$ is transitive if there exists $x\in X$ such that $\{R_nx:n\in\mathbb{Z}\}$ is dense in $X$. In the special case where $R_n=T^n$ for every $n\in\mathbb{Z}$, for some homeomorphism $T\colon X\to X$, we just say that $T$ is transitive.

Let $(X,S_1,\dots,S_k)$ and $(Y,S_1,\dots,S_k)$ be $\mathbb{Z}^k$-systems. By a slight abuse of notation, we use the same symbols $S_1,\dots,S_k$ for the transformations acting on both $X$ and $Y$. We say that $Y$ is a \emph{factor} of $X$, or equivalently that $X$ is an \emph{extension} of $Y$, if there exists a continuous onto map $\pi\colon X\to Y$, called a \emph{factor map}, such that $\pi\circ S_i=S_i\circ \pi$ for every $1\leq i\leq k$. A factor map $\pi:X\to Y$ is \emph{almost one-to-one} if there exists a dense G$_\delta$ subset $\Omega$ of $X$ such that for any $x\in \Omega$, $\pi^{-1}(\pi(x))=\{x\}$.

\subsection{Nilmanifolds and nilsystems}

We refer to {\cite[Chapters 10 and 11]{Host_Kra_nilpotent_structures_ergodic_theory:2018}} for the material discussed in this section.

Let $G$ be a $d$-step nilpotent Lie group and $\Gamma$ a discrete cocompact subgroup of $G$. The compact manifold $X = G/\Gamma$ is called a {\em $d$-step nilmanifold}. The group $G$ acts on $X$ by left transformations, and we denote this action by $(g,x)\mapsto gx$. The Haar measure $\mu$ of $X$ is the unique probability measure on $X$ invariant under this action. Let $\tau\in G$ and $T$ be the transformation $x\mapsto \tau x$ of $X$. Then $(X,\mu,T)$ is called a {\em $d$-step nilsystem}. An inverse limit of $d$-step nilsystems is called a $d$-step {\em pro-nilsystem}.

A sequence $(g(n))_{n\in\mathbb{Z}}$ in $G$ is called {\em polynomial} if it has the form $g(n) = \tau_1^{p_1(n)}\tau_2^{p_2(n)}\dots\tau_m^{p_m(n)}$, where $\tau_1,\dots,\tau_m\in G$ and $p_1,\dots,p_m\in \Z[t]$. Given such a polynomial sequence $g$, we consider the corresponding sequence of transformations of $X$ given by $x\mapsto g(n)x$. Then the following two properties are equivalent: the sequence $(g(n))_{n}$ is transitive on $X$, and it is ergodic with respect to the Haar measure $\mu$.

\subsection{Dynamical cubes}

Let $d \in\N$ be an integer, and write $[d] = \{1, 2, \dots, d\}$. We view an element of $\{0,1\}^{d}$, the Euclidean cube, either as a sequence $\epsilon = (\epsilon_{1}, \dots, \epsilon_{d})$ of 0's and 1's; or as a subset of $[d]$. A subset $\epsilon$ corresponds to the sequence $(\epsilon_{1}, \dots, \epsilon_{d}) \in \{0,1\}^{d}$ such that $i \in \epsilon$ if and only if $\epsilon_{i} = 1$ for $i \in [d]$. If $X$ is a set, we denote $X^{2^{d}}$ by $X^{[d]}$ and we write a point $\bx \in X^{[d]}$ as $\bx = (x_{\epsilon} : \epsilon \subseteq [d])$. We can isolate the first coordinate, writing $X^{[d]}_{*} = X^{2^{d}-1}$ and writing a point $\bx\in X^{[d]}$ as $\bx = (x_{\emptyset},\bx_{*})$, where $\bx_{*} = (x_{\epsilon}: \epsilon\neq \emptyset)\in X^{[d]}_{*}$. For a point $x\in X$ we let $x^{[d]}\in X^{[d]}$ and $x^{[d]}_{*}\in X^{[d]}_{*}$ be the diagonal points all of whose coordinates are $x$.

Let $(X,S_1,\dots,S_k)$ be a $\Z^k$-system. The {\em $d$-dimensional face cube group}, denoted by $\cF^{[d]}(S_1,\dots,S_k)$, is the subgroup of homeomorphisms of $X^{[d]}$ given by \begin{align*}
    \{(S_{\sum_{i\in \epsilon} n_i}: \epsilon\subseteq[d]): n_1,\dots,n_d\in \Z^{k}\}.
\end{align*}

When there is no ambiguity, we write $\cF^{[d]}$ instead of $\cF^{[d]}(S_1,\dots,S_k)$.

The set of {\em dynamical cubes of dimension $d$} associated with $(X,S_1,\dots,S_k)$ is defined by \begin{align*}
    \bQ^{[d]}(X,S_1,\dots,S_k) = \overline{\bigcup_{x\in X} \cF^{[d]}x^{[d]}}.
\end{align*}

When there is no ambiguity, we write $\bQ^{[d]}(X)$ instead of $\bQ^{[d]}(X,S_1,\dots,S_k)$.

\subsection{The regionally proximal relations}
Let $(X, S_1,\dots,S_k)$ be a $\Z^k$-system and $d\geq 1$ be an integer. A pair $(x,y)\in X\times X$ is said to be {\em regionally proximal of order $d$} if there are sequences $(f_{i})_{i\in\N}\subseteq \cF^{[d]}$, $(x_{i})_{i\in\N},(y_{i})_{i\in\N}\subseteq X$, and a point $a_{*}\in X^{[d]}_{*}$ such that $(f_{i}x^{[d]}_{i},f_{i}y_{i}^{[d]})\to (x,a_{*},y,a_{*})$. The set of regionally proximal pairs of order $d$ is denoted by $\RP^{[d]}(X,S_1,\dots,S_k)$ and is called the {\em regionally proximal relation of order $d$}. When there is no ambiguity, we write $\RP^{[d]}(X)$ instead of $\RP^{[d]}(X,S_1,\dots,S_k)$.

The relation $\RP^{[d]}(X)$ is a closed and invariant relation. Moreover,\begin{align*}
     \cdots \subseteq \RP^{[d+1]}(X)\subseteq \RP^{[d]}(X)\subseteq \cdots \subseteq \RP^{[1]}(X). 
\end{align*}

The following theorem was proved in \cite{Host_Kra_Maass_nilstructure:2010} in the distal case and in \cite{Shao_Ye_regionally_prox_orderd:2012} for general minimal systems. An alternative proof of the first assertion is given in \cite{Alvarez_Donoso_cube_struct_univ_nil_applications:2025}.

\begin{theorem}
    Let $(X,S_1,\dots,S_k)$ be a minimal $\Z^{k}$-system and let $d\in\N$. Then\begin{enumerate}
        \item $\RP^{[d]}(X)$ is an equivalence relation.
        \item $(x,y)\in \RP^{[d]}(X)$ if and only if $(x,y^{[d+1]}_*)\in \bQ^{[d+1]}(X)$ if and only if $(x,y^{[d+1]}_*)\in \overline{\cF^{[d+1]}x^{[d+1]}}$.
    \end{enumerate} 
\end{theorem}

In particular, the quotient $X/\RP^{[d]}(X)$ is well defined. We usually denote this quotient by $X_d$ and call it the {\em $d$-step pro-nilfactor}. The following structure theorem was proved in \cite{Host_Kra_Maass_nilstructure:2010} for $\mathbb{Z}$-systems and in \cite{Gutman_Manners_Varju_nilspaces_III:2020} for $\mathbb{Z}^k$-systems.

\begin{theorem}
    Let $(X,S_1,\dots,S_k)$ be a minimal $\Z^{k}$-system and let $d\in\N$. Then $X$ is a $d$-step pro-nilsystem if and only if $\RP^{[d]}(X)=\Delta$.
\end{theorem}

In \cite{Dong_Donoso_Maass_Shao_Ye_infinite_step_nil:2013}, the authors introduced the relation $\RP^{[\infty]}(X)$, defined by $\RP^{[\infty]}(X) = \bigcap_{d\in\N} \RP^{[d]}(X)$. We usually denote the quotient $X/\RP^{[\infty]}(X)$ by $X_\infty$ and call it the {\em $\infty$-step pro-nilfactor}. They proved the following analogous structure theorem.

\begin{theorem}\label{thm: str_thm}
    Let $(X,S_1,\dots,S_k)$ be a minimal $\Z^{k}$-system. Then $X$ is a pro-nilsystem if and only if $\RP^{[\infty]}(X)=\Delta$.
\end{theorem}

\subsection{O-diagram}

The following is a classical theorem stating that every factor map between minimal metric systems can be lifted to an open factor by almost one-to-one modifications.

\begin{theorem}[See {\cite[Chapter VI]{deVries_elements_topological_dynamics:1993}}]\label{thm: AG_diagram}
    Given a factor $\pi\colon (X, S_1,\dots,S_k) \to (Y, S_1,\dots,S_k)$ between minimal $\Z^k$-systems, there exists a commutative diagram of factors (called an O-diagram)
    \[\begin{tikzcd}
	X && {X^*} \\
	\\
	{Y} && {Y^*}
	\arrow["\pi"', from=1-1, to=3-1]
	\arrow["{\theta^*}"', from=1-3, to=1-1]
	\arrow["{\pi^*}", from=1-3, to=3-3]
	\arrow["\theta", from=3-3, to=3-1]
\end{tikzcd}\]

    such that $\theta$ and $\theta^*$ are almost one-to-one factors and $\pi^*$ is an open factor.
\end{theorem}

\subsection{Families of $d$-tuples}

Let $d,s\in\N$. Given $d$ ordered families of polynomials\begin{align*}
    \cP_1=(p_{1,1},\ldots,p_{1,s}) ,\ldots, \cP_d=(p_{d,1},\ldots,p_{d,s})
\end{align*}

we define an \emph{ordered family of $s$ polynomial $d$-tuples} as follows:\begin{align*}
(\cP_1,\ldots,\cP_d)=\big((p_{1,1},\ldots,p_{d,1}),\ldots,(p_{1,s},\ldots,p_{d,s})\big).
\end{align*}

If $\mathbf{p}=(p_1,\ldots,p_d)$ is a polynomial $d$-tuple, we define its \emph{degree} by $\deg(\mathbf{p}) = \max_{1\leq i\leq d}\deg(p_i)$.  Similarly, the \emph{degree of the family} $\cP=(\cP_1,\ldots,\cP_d)$ is defined by $\deg(\cP)=\max_{\mathbf{p}\in \cP} \deg(\mathbf{p})$.

We remove every $d$-tuple that consists entirely of constant polynomials, as well as all repeated $d$-tuples. Moreover, we shall always assume that $p(0)=0$ for every polynomial $p$ belonging to one of the families $\cP_1,\ldots,\cP_d$.

In \cite{Chu_Frantzikinakis_Host_ergodic_averages_distinct_degree:2011}, the authors introduced a class of families of polynomial $d$-tuples to study characteristic factors for multiple ergodic averages along polynomials of distinct degrees. For the topological counterpart, we introduce a more restrictive class, which is a subclass of the one defined in \cite{Chu_Frantzikinakis_Host_ergodic_averages_distinct_degree:2011} and will be used throughout the sequel.

\begin{definition}

Let 

\begin{align*}
    \cP=
\big(
(p_{1,1},\ldots,p_{d,1}),
\ldots,
(p_{1,s},\ldots,p_{d,s})
\big)
\end{align*}

be an ordered family of polynomial $d$-tuples. We say that $\cP$ is {\em nice} if
\begin{enumerate}
\item
$\deg(p_{1,1})\geq \deg(p_{1,j})$ for $j=1,\ldots,s$ ;
\medskip
\item
$\deg(p_{1,1})>\deg(p_{i,j})$ for $i=2,\ldots,d$, $j=1,\ldots,s$ ;
\medskip
\item
$\deg(p_{1,1}-p_{1,j})>\deg(p_{i,1}-p_{i,j})$ for $i=2,\ldots,d$,
$j=2,\ldots,s$.
\end{enumerate} 

 For $1\leq r\leq s$, write \begin{align*}
    \cP^{[r]}=\big((p_{1,r},\ldots,p_{d,r}),\ldots,(p_{1,s},\ldots,p_{d,s})\big).
\end{align*}

We define the notion of a {\em very nice family} recursively on $d$. If $d=1$, we say that $\cP$ is very nice if $\cP^{[r]}$ is nice for every $1\leq r\leq s$. Suppose now that $d\geq2$. We say that $\cP$ is very nice if $\cP$ is nice and one of the following conditions holds:\begin{enumerate}
    \item If $p_{1,j}\neq0$ for every $1\leq j\leq s$, then $\cP^{[r]}$ is nice for every $2\leq r\leq s$.
    \item If $p_{1,j}=0$ for some $1\leq j\leq s$, let $j_0$ be the smallest integer such that $p_{1,j_0}=0$. Then $p_{1,k}=0$ for every $j_0\leq k\leq s$, $p_{2,j_0}\neq 0$, the family $\mathcal P^{[r]}$ is nice for every $2\leq r<j_0$, and the family
\begin{align*}
    \big(
(p_{2,j_0},\ldots,p_{d,j_0}),
\ldots,
(p_{2,s},\ldots,p_{d,s})
\big)
\end{align*}

is very nice.
\end{enumerate}

\end{definition}

The following examples illustrate the recursive condition in the definition and the distinction between nice and very nice families.

\begin{example}
    The family \begin{align*}
    \big( (n^3,n^2,n),(0,n^2,n),(0,n^2+n,n) \big)
\end{align*}

is very nice. Indeed, the first zero polynomial in the first coordinate occurs at the second tuple. After removing the first coordinate and restricting to the corresponding reduced family, we obtain
\begin{align*}
    \big(
        (n^2,n),
        (n^2+n,n)
    \big),
\end{align*}
which is a very nice family of polynomial $2$-tuples. Thus, the recursive condition in $(2)$ of the definition is satisfied.
\end{example}

\begin{example}
    The family\begin{align*}
        \big( (n^4,0,0),(0,n,n^2),(0,n^2,0) \big)
    \end{align*}

    is a nice family, but it is not very nice. Indeed, the first zero polynomial in the first coordinate occurs at the second tuple. After removing the first coordinate and restricting to the corresponding reduced family, we obtain\begin{align*}
        \big( (n,n^2),(n^2,0)\big),
    \end{align*}

    which is not a nice family of polynomial $2$-tuples. Consequently, the recursive condition in $(2)$ of the definition fails, and hence the family is not very nice.
\end{example}

\section{Topological characteristic factors for a transitive transformation}

In this section, we prove a refinement of {\cite[Theorem 4.2]{Glasner_Huang_Shao_Weiss_Ye_Topological_characteristic_factors:2020}} for transitive transformations in minimal $\mathbb{Z}^k$-systems.

Using the same proof as in {\cite[Theorem 4.16]{Glasner_Gutman_Ye_higher_regionallyproximal_general_groups:2018}}, we obtain the following lemma.

\begin{lemma}\label{lemma: Qd_x_Fdx}
    Let $(X,S_1,\dots,S_k)$ be a minimal $\Z^k$-system, let $T\in \langle S_1,\dots,S_k \rangle$ be a transitive transformation and let $d\in \N$. Then there exists a dense G$_\delta$ $\langle S_1,\dots,S_k\rangle$-invariant subset $\Omega$ of $X$ such that\begin{align*}
        \{\bx\in \bQ^{[d+1]}(X,T): \bx_\emptyset = x\} = \overline{\cF^{[d+1]}(T)x^{[d+1]}}
    \end{align*}

    for each $x\in \Omega$.
\end{lemma}

\begin{lemma}[{\cite[Theorem 4.4]{Shao_Xu_saturation_R_flows:2025}}]\label{lemma: cube_G_T}
    Let $(X,S_1,\dots,S_k)$ be a minimal $\Z^k$-system, let $T\in \langle S_1,\dots,S_k \rangle$ be a transitive transformation and let $d\in \N$. Then $\bQ^{[d]}(X,S_1,\dots,S_k)=\bQ^{[d]}(X,T)$.
\end{lemma}

A consequence of {\cite[Lemma 4.1 and Theorem 4.3]{Shao_Xu_saturation_R_flows:2025}} is the following lemma.\begin{lemma}\label{lemma: RPd_G_T}
    Let $(X,S_{1},\dots,S_{k})$ be a minimal $\Z^{k}$-system, let $T\in \langle S_{1},\dots,S_{k}\rangle$ be a transitive transformation and let $d,n\in \N$. Then \begin{align*}
        \RP^{[d]}(X,S_1,\dots,S_k) = \RP^{[d]}(X,T) = \RP^{[d]}(X,T^n).
    \end{align*}
\end{lemma}

Recall that a collection $\cF$ of subsets of $\Z$ is a {\em family} if it is hereditarily upward, i.e., $F_1\subset F_2$ and $F_1\in \cF$ imply that $F_2\in \cF$. For a family $\cF$, its {\em dual} is the family $\cF^{*} =\{ F\subseteq \Z: F\cap F'\neq \emptyset \text{ for all }F'\in \cF\}$. The collection of all sets containing finite IP-sets of arbitrarily large length is denoted by $\cF_{fip}$.

The following lemma gives a characterization of certain pairs in $\RP^{[\infty]}$ for a transitive transformation. It is similar to {\cite[Theorem 7.2.7]{Huang_Shao_Ye_nilbohr_automorphy:2016}}, but its proof is different from the proof given there.

\begin{lemma}
    Let $(X,S_1,\dots,S_k)$ be a minimal $\Z^k$-system and let $T\in \langle S_1,\dots,S_k \rangle$ be a transitive transformation. Then there exists a dense G$_\delta$ $\langle S_1,\dots,S_k\rangle$-invariant subset $\Omega$ of $X$ such that  each $x\in \Omega$ satisfies the following: if $(x,y)\in \RP^{[\infty]}(X,T)$, then for any neighborhood $U$ of $y$, for any $d\in \N$, for any minimal system $(Y,T)$ and any nonempty open subset $V$ of $Y$, there exists\begin{align*}
        n\in \{k\in \Z: T^{k}x\in U\}
    \end{align*} such that\begin{align*}
        V \cap T^{-n}V\cap \dots\cap T^{-dn} V \neq \emptyset.
    \end{align*}  
\end{lemma}

\begin{proof}
    For $d\in\mathbb{N}$, let $\Omega_d$ be the dense G$_\delta$ subset of $X$ given by \cref{lemma: Qd_x_Fdx}. Define $\Omega = \cap_{d\in\N}\Omega_{d}$. By \cref{lemma: cube_G_T} and \cref{lemma: RPd_G_T}, we have $(x,y^{[d+1]}_{*})\in \bQ^{[d+1]}(X,T)$. Since $x\in \Omega$, it follows that $(x,y^{[d+1]}_{*})\in \overline{\cF^{[d+1]}(T)x^{[d+1]}}$. Therefore,\begin{align*}
        \{k\in \Z: T^{k}x\in U\}\in \cF_{fip}.
    \end{align*}
    
    Let $\mu$ be an ergodic measure for $(Y,T)$. Since $(Y,T)$ is minimal, $\mu$ has full support, in particular, $\mu(V)>0$. Hence, by {\cite{Furstenberg_Katznelson85}}, we have\begin{align*}
        \{n\in\Z: \mu(V \cap T^{-n}V\cap \dots\cap T^{-dn} V)>0 \} \in \cF_{fip}^{*}.
    \end{align*}

    Combining this with $\{k\in \Z: T^{k}x\in U\}\in \cF_{fip}$, we conclude the desired result.
\end{proof}

Given a $\Z^{k}$-system $(X,S_1,\dots,S_k)$, $T\in \langle S_1,\dots,S_k\rangle$, and an integer $d\geq 1$, set \begin{align*}
    \sigma_{d} &= \{(S_n,S_n,\dots,S_n): n\in \Z^k\},\\
    \tau_{d} &=  T\times T^2\times \dots\times T^d,\\
    N_{d}(X,T) &= \overline{\cO}( \Delta^{(d)}(X),\tau_d).
\end{align*}

Here, for $A\subseteq X$, we write $\Delta^{(d)}(A)=\{x^{(d)}:x\in A\}\subseteq X^{d}$, where $x^{(d)} = (x,\dots,x)\in X^{d}$. When there is no ambiguity, we write $N_d(X)$ instead of $N_d(X,T)$.

\begin{lemma}
    Let $(X,S_1,\dots,S_k)$ be a minimal $\Z^k$-system, let $X_0$ be a dense G$_\delta$ $\langle S_1,\dots,S_k\rangle $-invariant subset of $X$, let $T\in \langle S_1,\dots,S_k \rangle$ be a transitive transformation and let $d\in \N$. Then the system $(N_d(X),\langle \sigma_{d},\tau_{d}\rangle)$ is minimal and, for each $j\in \{1,\dots,d\}$, the set\begin{align*}
        M_j=\{(x_1,\dots,x_d)\in N_{d}(X,T): x_j\in X_{0} \text{ and } (x_1,\dots,x_d)\text{ is a }\tau_{d}^{-1}(T^{(d)})^{j}\text{-minimal point}\}
    \end{align*}

    is dense in $N_{d}(X)$.
\end{lemma}

\begin{proof}
    The minimality of $(N_d(X),\langle \sigma_d,\tau_d\rangle)$ follows from the same proof as {\cite[Proposition 1.55]{Glasner_ergodic_theory_joinings:2003}}. We now prove the second statement.

    Fix $j\in \{1,\dots,d\}$. Let $x\in X_0$ and let $(M,\tau_{d}^{-1}(T^{(d)})^{j})$ be a minimal subsystem of the system $(\overline{\cO}(x^{(d)},\tau_{d}^{-1}(T^{(d)})^{j}),\tau_{d}^{-1}(T^{(d)})^{j})$. Note that $\cup_{n\in \Z^k}(S_n\times \dots \times S_n) M$ is a $\langle \sigma_{d},\tau_{d}\rangle$-invariant subset of $M_j$. Therefore, by the minimality of $(N_d(X),\langle \sigma_d,\tau_d\rangle)$, this set is dense in $N_d(X)$.
\end{proof}

Using the preceding lemmas and following the proof of {\cite[Theorem 4.2]{Glasner_Huang_Shao_Weiss_Ye_Topological_characteristic_factors:2020}} without any further changes, we obtain the following topological characteristic factor theorem.

\begin{theorem}
    Let $\pi\colon (X,S_1,\dots,S_k) \to (Y,S_1,\dots,S_k)$ be an open factor between minimal $\Z^k$-systems, and let $T\in \langle S_{1},\dots,S_{k}\rangle$ be a transitive transformation. If $X_\infty$ is a factor of $Y$, then $Y$ is a \emph{$d$-step topological characteristic factor} of $X$ for $T$ for all $d\in \N$, that is, there exists a dense G$_\delta$ subset $\Omega$ of $X$ such that for each $x\in \Omega$,\begin{align*}
       (\pi^{(d)})^{-1}(\pi^{(d)}( \overline{\cO}(x^{(d)},\tau_d)))= \overline{\cO}(x^{(d)},\tau_d).
    \end{align*}
\end{theorem}

Using this theorem and following the proof of {\cite[Lemma 2.19]{Qiu_poly_orbits_tot_minimal:2023}}, we obtain the following linear recurrence result.

\begin{lemma}
    Let $\pi\colon (X,S_1,\dots,S_k) \to (Y,S_1,\dots,S_k)$ be an open factor between minimal $\Z^k$-systems, and let $T\in \langle S_{1},\dots,S_{k}\rangle$ be a transitive transformation. If $X_\infty$ is a factor of $Y$, then for any distinct nonzero integers $a_1,\dots,a_s$, there is a dense G$_\delta$ subset $\Omega$ of $X$ such that for any open subsets $V_0,V_1,\dots,V_s$ of $X$ with $\cap_{i=0}^{s}\pi(V_i)\neq \emptyset$ and any $z\in V_0\cap \Omega$ with $\pi(z)\in\cap_{i=0}^{s}\pi(V_i)$, there exists some $A\in \cF_{fip}$ such that $T^{a_i n}z\in V_i$, for $i=1,\dots,s$ and $n\in A$.
\end{lemma}

To add polynomial conditions to this result, we use the following lemma.

\begin{lemma}[{\cite[Proposition 3.14]{Leibman05a}}]
    Let $(X,T_1,\dots,T_d)$ be a nilsystem, $x\in X$ and an open neighborhood $U$ of $x$. For any nonconstant integer polynomials $p_1,\dots,p_d$ with $p_i(0)=0$ for every $i=1,\dots,d$, there exists another nilsystem $(Y,T)$ with $y\in Y$ and an open neighborhood $V$ of $y$ such that\begin{align*}
        \{n\in \Z: T^n y\in V\} \subseteq\{n\in\Z: T_1^{p_1(n)}\dots T_d^{p_d(n)}x\in U\}.
    \end{align*}
\end{lemma}

The proof of {\cite[Lemma 4.3]{Qiu_poly_orbits_tot_minimal:2023}} applies without change in our setting, using the previous two lemmas. We thus obtain the following stronger statement. 

\begin{lemma}\label{lemma: linear stronger}
    Let $\pi\colon (X,S_{1},\dots,S_{k})\to (Y,S_1,\dots,S_k)$ be an open factor between minimal $\Z^{k}$-systems, and let $T_1,\dots,T_d\in \langle S_{1},\dots,S_{k}\rangle$ be transitive transformations. If $Y$ is an almost one-to-one extension of $X_\infty$, then for any distinct nonzero integers $c_1,\dots, c_s$, any family of polynomial $d$-tuples $\cC$, and open subsets $V_0,V_1,\dots,V_s$ of $X$ with $\cap_{i=0}^{s}\pi(V_i)\neq \emptyset$, there exist infinitely many $n\in \N$ for which there exists $z\in V_0$ such that $T_1^{c_i n}z\in V_i$ for $1\leq i \leq s$ and $T_1^{q_1(n)}\dots T_d^{q_d(n)}\pi(z)\in \cap_{i=0}^{s}\pi(V_i)$ for $(q_1,\dots,q_d)\in \cC$.
\end{lemma}

Finally, it is worth noting that the results of this section also allow one to extend the saturation theorem for polynomial iterates in \cite{Qiu_poly_orbits_tot_minimal:2023,Ye_Yu_polynomial_saturation:2025} to transitive transformations in minimal $\Z^k$-systems.

\section{Topological characteristic factor for a very nice family}

In this section, following the approach of \cite{Chu_Frantzikinakis_Host_ergodic_averages_distinct_degree:2011,Qiu_poly_orbits_tot_minimal:2023}, we prove the following theorem.\begin{theorem}\label{thm: TFC_nice_pol}
    Let $\pi\colon (X,S_1,\dots,S_k) \to (Y,S_1,\dots,S_k)$ be an open factor between minimal $\Z^k$-systems, and $T_1,\dots,T_d\in \langle S_{1},\dots,S_{k}\rangle$ be transitive transformations. If $Y$ is an almost one-to-one extension of $X_\infty$, then for any open subsets $V_{0},V_{1},\dots,V_{s}$ of $X$ with $\cap_{i=0}^{s}\pi(V_i)\neq \emptyset$ and any very nice family of polynomial $d$-tuples $\big( (p_{1,1},\dots,p_{d,1}),\dots,(p_{1,s},\dots,p_{d,s})\big)$, there exists some $n\in \Z$ such that\begin{align*}
        V_{0}\cap T_{1}^{-p_{1,1}(n)}\cdots T_{d}^{-p_{d,1}(n)}V_{1}\cap \dots \cap T_{1}^{-p_{1,s}(n)}\cdots T_{d}^{-p_{d,s}(n)}V_{s}\neq \emptyset.
    \end{align*}
\end{theorem}

\subsection{Types of families and the van der Corput operation}

To prove \cref{thm: TFC_nice_pol}, we use a PET induction argument, introduced in \cite{Bergelson_WM_PET:1987}. Following \cite{Chu_Frantzikinakis_Host_ergodic_averages_distinct_degree:2011,Bergelson_Leibman96}, we define in this subsection the type of a family of polynomial tuples and an operation that reduces its type.

\subsubsection{Definition of type}

We fix an integer $D\geq 1$ and consider only families of polynomial $d$-tuples of degree at most $D$.

Two nonconstant polynomials $p,q\in\Z[t]$ are said to be \emph{equivalent}, and we write $p\sim q$, if they have the same degree and the same leading coefficient.

Let $(\cP_1,\dots,\cP_d) = \big( (p_{1,1},\dots,p_{d,1}),\dots,(p_{1,s},\dots,p_{d,s})\big)$ be an ordered family of $s$ polynomial $d$-tuples. For $i=1,\ldots, d$, we define $\widehat{\cP}_i$ to be the following set (possibly empty)\begin{align*}
    \widehat{\cP}_i=\{ \text{nonconstant } p_{i,m} \in \cP_i\colon p_{i',m} \text{ is constant for } i'<i\}.
\end{align*}

In particular, $\widehat{\cP}_1$ is the collection of nonconstant polynomials of $\cP_1$.

For $i=1,\ldots, d$ and $j=1,\ldots, D$, we let $w_{i,j}$ be the number of distinct non-equivalent classes of polynomials of degree $j$ in the family $\widehat{\cP}_i$. We define the \emph{(matrix) type} of the family $(\cP_1,\ldots, \cP_d)$ to be the matrix\begin{align*}
    W=\begin{pmatrix}
w_{1,D}& \ldots &  w_{1,1}\\ w_{2,D}& \ldots& w_{2,1}\\ \vdots & \ldots &\vdots \\
w_{d,D}& \ldots& w_{d,1}
\end{pmatrix}.
\end{align*}

We order these  types  lexicographically: Given two $d\times D$ matrices $W=(w_{i,j})$ and $W'=(w'_{i,j})$, we say that the first is bigger than the second, and write $W>W'$, if $w_{1,D}>w'_{1,D}$, or $w_{1,D}=w'_{1,D}$ and $w_{1,D-1}>w'_{1,D-1}$, $\ldots$, or $w_{1,i}=w'_{1,i}$ for $i=1,\ldots,D$ and $w_{2,D}>w'_{2,D}$, and so on.

\subsubsection{The van der Corput operation}
Given a family $\cP=(p_1,\ldots,p_s)$ of polynomials, $p\in \Z[t]$, and $h\in\Z$, we define
\begin{align*}
    (\partial_h p)(n) = p(n+h)-p(h),\, \partial_h \cP = (\partial_h p_1,\dots, \partial_h p_s)\, \text{ and } \,\cP -p=(p_1-p,\dots,p_s-p).
\end{align*}

Given a family $(\cP_1,\ldots,\cP_d)$ of $d$-tuples of polynomials, $(p_1,\ldots,p_d)$ a $d$-tuple of polynomials, and $h_1,\dots,h_\ell\in\N$, define\begin{align*}
    (p_1,\ldots, p_d,h_1,\dots,h_\ell)\vdc(\cP_1,\ldots, \cP_d) =(\tilde{\cP}_{1,h_1,\dots,h_\ell},\dots ,\tilde{\cP}_{d,h_1,\dots,h_\ell})^*,
\end{align*}

 where, for each $i=1,\dots,d$,\begin{align*}
     \tilde{\cP}_{i,h_1,\dots,h_\ell}=(\partial_{h_1+h_2+\dots+h_\ell}\cP_i-p_i,\partial_{h_2+\dots+h_\ell}\cP_i-p_i,\dots,\partial_{h_\ell}\cP_i-p_i,\cP_i-p_i).
 \end{align*}

 Here $^*$ denotes the operation of removing from the resulting family all $d$-tuples consisting entirely of constant polynomials and repeated $d$-tuples. Notice that if $(\cP_1,\ldots, \cP_d)$ is a degree $D$ family containing $s$ polynomial $d$-tuples, then the family $(p_1,\ldots, p_d,h_1,\dots,h_\ell)\vdc(\cP_1,\ldots, \cP_d)$ has degree at most $D$ and contains at most $(\ell+1)s$ polynomial $d$-tuples for every $h_1,\dots,h_\ell\in\N$,.

In \cite[Lemma 5.4]{Chu_Frantzikinakis_Host_ergodic_averages_distinct_degree:2011}, it was proved for $\ell=1$ that, after choosing a suitable tuple from a nice family, the van der Corput operation produces another nice family of strictly smaller type. However, even if the original family is very nice, the resulting family need not be very nice with the order given by the van der Corput operation.

For example, consider the very nice family\begin{align*}
    \cP=\big( (n^3,n^2,n),(0,n^2,n),(0,n^2+n,n) \big).
\end{align*}

The family $(0,n^2,n,h)\vdc \cP$ is \begin{align*}
    \big( (n^3+3hn^2+3h^2n,2hn,0),(0,2hn,0),(0,(2h+1)n,0),(n^3,0,0),(0,n,0) \big).
\end{align*}

This family is not very nice. However, if we first place the tuples with nonzero first coordinate and then order the remaining tuples by decreasing degree of their second coordinate, we obtain\begin{align*}
    \big( (n^3+3hn^2+3h^2n,2hn,0),(n^3,0,0),(0,2hn,0),(0,(2h+1)n,0),(0,n,0) \big),
\end{align*}

which is very nice.

The following lemma shows that such a reordering is possible for every very nice family.

\begin{lemma}\label{lemma: reduce_type}
Let $(\cP_1,\ldots, \cP_d)$ be a family with $\deg((\cP_1,\dots,\cP_d))\geq 2$. 

Then there exists $(p_1,\ldots,p_d)\in (\cP_1,\ldots,\cP_d)$ such that for every $\ell\in\mathbb N$ and every large enough $h_1,\dots,h_\ell\in\N$, the family  $(p_1,\ldots,p_d,h_1,\dots,h_\ell)\vdc(\cP_1,\ldots, \cP_d)$ has strictly smaller type than $(\cP_1,\ldots, \cP_d)$. Moreover, if $(\cP_1,\ldots,\cP_d)$ is very nice, then $(p_1,\ldots,p_d,h_1,\dots,h_\ell)\vdc(\cP_1,\ldots, \cP_d)$ can be reordered so as to be very nice.
\end{lemma}
\begin{proof}
We choose $(p_1,\dots,p_d)$ as follows. If $\widehat{\cP}_i\neq\emptyset$ for some $i\geq 2$, let $i$ be maximal with this property and choose $(p_1,\dots,p_d)$ such that $p_1=\cdots=p_{i-1}=0$ and $p_i$ has minimal degree in $\widehat{\cP}_i$. Then, for every $h_1,\dots,h_\ell\in\mathbb N$, the first $i-1$ rows of the type matrix are unchanged, and the $i$-th row will get reduced. Hence, the resulting family has strictly smaller type.

Suppose now that the families $\widehat{\cP}_2,\dots,\widehat{\cP}_d$ are empty. If $\widehat{\cP}_1=\{p_{1,1}\}$, or if every polynomial in $\cP_1$ is equivalent to $p_{1,1}$, we choose $(p_1,\dots,p_d)=(p_{1,1},\dots,p_{d,1})$. Otherwise, we choose a tuple whose first coordinate has minimal degree among the polynomials in $\cP_1$ that are not equivalent to $p_{1,1}$. In each case, the first row of the type matrix of $(p_1,\ldots,p_d,h_1,\dots,h_\ell)\vdc(\cP_1,\ldots, \cP_d)$ is smaller than that of $(\cP_1,\dots,\cP_d)$. 

Suppose now that $(\cP_1,\dots,\cP_d)$ is nice. We show that $(p_1,\ldots,p_d,h_1,\dots,h_\ell)\vdc(\cP_1,\ldots, \cP_d)$ is also nice. In the proof of {\cite[Lemma 5.4]{Chu_Frantzikinakis_Host_ergodic_averages_distinct_degree:2011}}, it was shown that the choice of $(p_1,\dots,p_d)$ guarantees that, for all but finitely many $h\in\mathbb N$,\begin{align*}
\deg(\partial_h p_{1,1}-p_1)
&\geq \max\{\deg(p_{1,m}-p_1),\deg(\partial_h p_{1,m}-p_1)\},
&&m=2,\dots,s;\\
\deg(\partial_h p_{1,1}-p_1)
&> \max\{\deg(p_{i,m}-p_i),\deg(\partial_h p_{i,m}-p_i)\},
&&i=2,\dots,d,\quad m=2,\dots,s;\\
\deg(\partial_h p_{1,1}-\partial_h p_{1,m})
&>\deg(\partial_h p_{i,1}-\partial_h p_{i,m}),
&&i=2,\dots,d,\quad m=2,\dots,s;\\
\deg(\partial_h p_{1,1}-p_{1,m})
&>\deg(\partial_h p_{i,1}-p_{i,m}),
&&i=2,\dots,d,\quad m=2,\dots,s.
\end{align*}

Therefore, to verify that $(p_1,\ldots,p_d,h_1,\dots,h_\ell)\vdc(\cP_1,\ldots, \cP_d)$ is nice it is enough to show that, for all sufficiently large $h_{1},\dots,h_{\ell}$, \begin{align*}
    \deg(\partial_{h_j+\dots+h_\ell} p_{1,1} - \partial_{h_{j'}+\dots+h_\ell} p_{1,m}) > \deg(\partial_{h_j+\dots+h_\ell} p_{i,1} - \partial_{h_{j'}+\dots+h_\ell} p_{i,m})
\end{align*}

for $2\leq i \leq d$, $1\leq m\leq s$ and $1\leq j<j'\leq \ell$. 

If $p_{1,1}\not\sim p_{1,m}$, we have \begin{align*}
            \deg(\partial_{h_j+\dots+h_{\ell}} p_{1,1} - \partial_{h_{j'}+\dots+h_{\ell}}p_{1,m}) = \deg(p_{1,1}).
        \end{align*}

        Since $\deg(p_{i,m}),\deg(p_{i,1})<\deg(p_{1,1})$, it follows that\begin{align*}
            \deg(\partial_{h_j+\dots+h_{\ell}} p_{i,1} - \partial_{h_{j'}+\dots+h_{\ell}}p_{i,m}) < \deg(\partial_{h_j+\dots+h_{\ell}} p_{1,1} - \partial_{h_{j'}+\dots+h_{\ell}}p_{1,m}).
        \end{align*}
        
        Suppose now that $p_{1,1}\sim p_{1,m}$. For every sufficiently large $h_j,h_{j+1},\dots,h_{\ell-1}$, we have \begin{align*}
            \deg(\partial_{h_j+\dots+h_{\ell}}p_{1,1} - \partial_{h_{j'}+\dots+h_{\ell}}p_{1,m}) = \deg(p_{1,1})-1. 
        \end{align*}

        Moreover,\begin{align*}
            \partial_{h_j+\dots+h_{\ell}} p_{i,1} - \partial_{h_{j'}+\dots+h_{\ell}}p_{i,m} =(\partial_{h_j+\dots+h_{\ell}} p_{i,1} -p_{i,1})+(p_{i,1}-p_{i,m}) - (\partial_{h_{j'}+\dots+h_{\ell}}p_{i,m} - p_{i,m})
        \end{align*}

        The first and third terms on the right-hand side have degrees at most $\deg(p_{i,1})-1$ and $\deg(p_{i,m})-1$, and hence strictly smaller than $\deg(p_{1,1})-1$. Furthermore, since $(\cP_1,\dots,\cP_d)$ is nice, it follows that the second term on the right side has degree strictly smaller than $\deg(p_{1,1}-p_{1,m})$. Thus,\begin{align*}
    \deg(\partial_{h_j+\dots+h_\ell} p_{1,1} - \partial_{h_{j'}+\dots+h_\ell} p_{1,m}) > \deg(\partial_{h_j+\dots+h_\ell} p_{i,1} - \partial_{h_{j'}+\dots+h_\ell} p_{i,m}).
\end{align*}

Thus the resulting family is nice. Removing constant and repeated tuples does not affect these properties.

Suppose that $(\cP_1,\dots,\cP_d)$ is very nice. We show that the tuples of $(p_1,\ldots,p_d,h_1,\dots,h_\ell) \vdc (\cP_1,\ldots, \cP_d)$ can be reordered so that the resulting family is very nice.

We first note two inequalities that follow from the definition of a very nice family that will be used. If $m\neq m'$ and $i$ is the smallest integer such that $p_{i,m}\neq p_{i,m'}$, then\begin{align*} \deg(p_{i,m}-p_{i,m'})> \deg(p_{k,m}-p_{k,m'}) \end{align*} for every $k>i$. Similarly, if $i$ is the first nonzero coordinate of $(p_{1,m},\dots,p_{d,m})$, then $\deg(p_{i,m})>\deg(p_{k,m})$ for every $k>i$.

For $1\leq i\leq d$, define\begin{align*}
    \cP_i' = \{(q_1,\dots,q_d)\in(p_1,\ldots,p_d,h_1,\dots,h_\ell)\vdc(\cP_1,\ldots, \cP_d): q_1 = \dots=q_{i-1}=0, q_i\neq 0 \}.
\end{align*}

Order the tuples in each $\cP_i'$ so that the degrees of their $i$-th coordinates are non-increasing, and then order $(p_1,\ldots,p_d,h_1,\dots,h_\ell)\vdc(\cP_1,\ldots, \cP_d)$ as $\cP_1',\cP_2',\dots,\cP_d'$. Denote the resulting family by $\cP'$, and write\begin{align*}
    \cP' = \big( (q_{1,1},\dots,q_{d,1}),\dots, (q_{1,t},\dots,q_{d,t}) \big).
\end{align*}

Let $(q_1,\dots,q_d)\in\cP_i'$, and suppose that\begin{align*}
    (q_1,\dots,q_d) =  (q_{1,m},\dots,q_{d,m})
\end{align*}

for some $1\leq m \leq t$. We will show that the family\begin{align*}
    \big( (q_{i,m},\dots,q_{d,m}),\dots,(q_{i,t},\dots,q_{d,t}) \big)
\end{align*}

is nice.

We first claim that, for all sufficiently large $h_1,\ldots,h_\ell$, if $(q_1,\dots,q_d),(q_1',\dots,q_d')$ are distinct $d$-tuples in $(p_1,\ldots,p_d,h_1,\dots,h_\ell)\vdc(\cP_1,\ldots, \cP_d)$ and $i$ is the smallest integer such that $q_i \neq q_i'$, then $\deg(q_i-q_i')>\deg(q_k-q_k')$ for every $k>i$. 

Suppose first that there exist $1\leq m,m' \leq s$ and $1\leq j,j' \leq \ell$ such that\begin{align*}
    q_i - q_i' = \partial_{h_j+\dots+h_{\ell}}p_{i,m} - \partial_{h_{j'}+\dots+h_{\ell}}p_{i,m'}.
\end{align*}

%for every $1\leq i \leq d$.

If $p_{i,m}\not\sim p_{i,m'}$, then\begin{align*}
    \deg(\partial_{h_{j}+\dots+h_{\ell}}p_{i,m} -\partial_{h_{j'}+\dots+h_{\ell}}p_{i,m'}) = \max\{\deg(p_{i,m}),\deg(p_{i,m'})\},
\end{align*}

and the claim follows from the defining properties of a very nice family.

Suppose that $p_{i,m}\sim p_{i,m'}$. If $j\neq j'$, then, for all sufficiently large $h_1,\ldots,h_\ell$,\begin{align*}
    \deg(\partial_{h_{j}+\dots+h_\ell}p_{i,m} -\partial_{h_{j'}+\dots+h_\ell}p_{i,m'}) = \deg(p_{i,m})-1.
\end{align*}

For $k>i$, we have\begin{align*}
    \partial_{h_{j}+\dots+h_{\ell}}p_{k,m}-\partial_{h_{j'}+\dots+h_{\ell}}p_{k,m'} =(\partial_{h_{j}+\dots+h_{\ell}}p_{k,m}-p_{k,m})+(p_{k,m}-p_{k,m'}) -(\partial_{h_{j'}+\dots+h_{\ell}}p_{k,m'}-p_{k,m'}).
\end{align*}

The first and third terms on the right-hand side have degree strictly smaller than $\deg(p_{i,m})-1$. Moreover, since $\cP$ is very nice, it follows that \begin{align*}
    \deg(p_{k,m}-p_{k,m'}) < \deg(p_{i,m}-p_{i,m'})\leq \deg(p_{i,m})-1.
\end{align*}

Thus the claim follows in this case. The case $j=j'$ follows directly from the defining inequalities of a very nice family. The case $q_i-q_i' = p_{i,m}-p_{i,m'}$ also follows directly from the defining inequalities of a very nice family.

Finally, suppose now that there exist $1\leq m,m' \leq s$ and $1\leq j \leq \ell$ such that\begin{align*}
    q_i - q_i' = \pm p_{i,m} \mp \partial_{h_j+\dots+h_{\ell}}p_{i,m'}
\end{align*}

Then the same argument as above proves the claim.

The same argument applied to $\partial_{h_j+\dots+h_\ell} (p_{1,m},\dots,p_{d,m}) - (p_1,\dots,p_d)$ shows that, if $i$ is the smallest integer such that $q_i \neq 0$, then $\deg(q_i)>\deg(q_k)$ for every $k>i$. Moreover, the same argument shows that, for sufficiently large $h_1,\ldots,h_\ell$, distinct tuples in $\cP_i'$ have distinct $i$-th coordinates.

Let $(q_1',\dots,q_d')$ be any tuple that follows $(q_1,\dots,q_d)$ in $\cP'$. Suppose first that $(q_1',\dots,q_d')\in\cP_i'$. By the choice of the ordering, $\deg(q_i)\geq\deg(q_i')$. Moreover, since $i$ is the smallest integer such that $q_i'\neq 0$, we have $\deg(q_i')>\deg(q_k')$ for every $k>i$. Thus, $\deg(q_i)>\deg(q_k')$ for every $k>i$. Furthermore, since $i$ is the smallest integer such that $q_i\neq q_i'$, it follows from the claim that \begin{align*}
    \deg(q_i-q_i')>\deg(q_k-q_k')
\end{align*}

for every $k>i$.

Suppose now that $(q_1',\dots,q_d')\in\cP_j'$ for some $j>i$. Then $q_i'=0$, and $i$ is the smallest integer such that $q_i'\neq q_i$. Therefore, it follows from the claim that \begin{align*}
    \deg(q_i) = \deg(q_i-q_i')>\deg(q_k-q_k')
\end{align*}

for every $k>i$. Thus, since $\deg(q_i)>\deg(q_k)$ for any $k>i$, it follows that  \begin{align*}
    \deg(q_i)>\max\{\deg(q_k),\deg(q_k-q_k')\}
\end{align*}
for every $k>i$. 

For $k>i$, we have $q_k' = (q_k' - q_k) +q_k$. Hence, $\deg(q_i)>\deg(q_k')$ for every $k>i$. 

Finally, in both cases, we have $\deg(q_i)>\deg(q_k')$ and $\deg(q_i-q_i')>\deg(q_k-q_k')$ for every $k>i$. Therefore, the family\begin{align*}
    \big( (q_{i,m},\dots,q_{d,m}),\dots,(q_{i,t},\dots,q_{d,t}) \big)
\end{align*} is nice. 

Applying the same argument successively to $\cP_2'\cup\dots\cup\cP_d',
\cP_3'\cup\dots\cup\cP_d',\dots,\cP_d'$ shows that the reordered family $\cP'$ is very nice. 
\end{proof}

\subsection{A stronger result}

Throughout this subsection, let $\pi\colon(X,S_{1},\dots,S_{k})\to (Y,S_{1},\dots,S_k)$ be an open factor between minimal $\Z^{k}$-systems and  let $T_{1},\dots,T_{d}\in \langle S_{1},\dots,S_{k} \rangle$ be transitive transformations. We assume that $Y$ is an almost one-to-one extension of $X_\infty$.

Let $\cP$ be a very nice family of $d$-tuples of polynomials, written as\begin{align*}
    \cP=(\cP_1,\cP_2,\dots,\cP_{d}) = \big((p_{1,1},p_{2,1},\dots,p_{d,1}),(p_{1,2},p_{2,2},\dots,p_{d,2}),\dots,(p_{1,s},p_{2,s},\dots,p_{d,s})\big)
\end{align*}  

and let $\cC$ be a family of $d$-tuples of polynomials. We say that $\pi$ has the property $\Lambda(\cP,\cC)$ if for any open subsets $V_{0},V_{1},\dots,V_{s}$ of $X$ with $\cap_{i=0}^{s} \pi(V_{i})\neq \emptyset$, there exist infinitely many $n\in\N$ for which there exists $z\in V_{0}$ such that\begin{enumerate}
    \item $T_{1}^{p_{1,i}(n)}T_2^{p_{2,i}(n)}\cdots T_{d}^{p_{d,i}(n)} z\in V_{i}$ for $1\leq i \leq s$;
    \item $T_1^{q_1(n)}T_2^{q_2(n)}\cdots T_{d}^{q_d(n)}\pi(z)\in \cap_{i=0}^{s} \pi(V_{i})$ for $(q_1,q_2,\dots,q_d)\in \cC$.
\end{enumerate}

To prove \cref{thm: TFC_nice_pol}, it is enough to prove the following result.\begin{theorem}
    Let $\cC$ be a family of $d$-tuples of polynomials and let $\cP$ be a very nice family of $d$-tuples of polynomials. Then $\pi$ has the property $\Lambda(\cP,\cC)$.
\end{theorem}

\begin{proof}
    Write\begin{align*}
        \cP=(\cP_1,\dots,\cP_d) &= \big((p_{1,1},\dots,p_{d,1}),\dots,(p_{1,s},\dots,p_{d,s})\big),\\
        \cC=(\cC_1,\dots,\cC_d) &= \big((q_{1,1},\dots,q_{d,1}),\dots,(q_{1,s'},\dots,q_{d,s'})\big).
    \end{align*}
    
    Suppose first that $\deg(p_{1,1})=1$. Since $\cP$ is nice, $\deg(p_{1,m})\leq 1$ and $p_{i,m}=0$ for $1\leq m \leq s$ and $2\leq i \leq d$. The desired conclusion then follows from \cref{lemma: linear stronger}.

    We may therefore assume that $\deg(p_{1,1})\geq 2$. Let $W$ be the matrix type of $\cP$. Assume that $\pi$ has the property $\Lambda(\cP',\cC')$ for any very nice family of $d$-tuples of polynomials $\cP'$ and any family of $d$-tuples of polynomials $\cC'$ whenever $\cP'$ has strictly smaller type than $W$.

    Let $V_{0},V_{1},\dots,V_{s}$ be open subsets of $X$ with $W_0= \cap_{m=0}^{s} \pi(V_{m})\neq \emptyset$. Replacing each $V_{m}$ by $V_{m}\cap \pi^{-1}(W_0)$, we may assume that $\pi(V_m)=W_0$ for $0\leq m \leq s$.

    Since $(X,S_{1},\dots,S_k)$ is minimal, there exists a finite set $F\subseteq \Z^{k}$ such that $X=\cup_{r\in F}S_{r} V_{m}$ for every $0\leq m\leq s$. Let $x\in X$ with $\pi(x)\in W_0$ and let \begin{align*}
        \{r\in F: \pi(x)\in S_{r}W_0\} = \{a_{1},\dots,a_{N}\}.
    \end{align*}
    
    Then we have $\pi^{-1}(\pi(x)) \subseteq \cap_{m=0}^{s}\cup_{j=1}^{N}S_{a_{j}}V_{m}$.
    
    Since $\pi$ is open, there is a $\delta>0$ such that $\pi^{-1}(B(\pi(x),\delta)) \subseteq \cap_{m=0}^{s}\cup_{j=1}^{N}S_{a_{j}}V_{m}$ and $B(\pi(x),\delta)\subseteq W_0\cap \left(\cap_{j=1}^{N} S_{a_{j}}W_0\right)$.

    The proof is divided into two cases.

\subsubsection*{Case 1: $\deg((p_{1,m},\dots,p_{d,m}))\geq 2$ for every $m=1,\dots,s$.}

Let $\eta=\delta/N$. Inductively we will construct infinitely many $n_1,\dots,n_{N}\in \N$ for which there exist $x_{1},\dots,x_{N}\in X$ such that for $1\leq j \leq \ell\leq N$,\begin{itemize}
    \item $\pi(x_{\ell})\in B(\pi(x),\ell\eta)$;
    \item $T_1^{p_{1,m}(n_j+\dots+n_\ell)}\cdots T_d^{p_{d,m}(n_j+\dots+n_\ell)} x_{\ell}\in S_{a_{j}}V_{m}$ for $1\leq m\leq s$;
    \item $ T_1^{q_1(n_{j}+\dots+n_{\ell})}\cdots T_d^{q_d(n_{j}+\dots+n_{\ell})} \pi(x_{\ell})\in B(\pi(x),\ell\eta)$ for $(q_1,\dots,q_d)\in \cC$.
\end{itemize}

Once this construction is completed, the desired result follows.

\textbf{Step 1:} Let $I_{1} = \pi^{-1} (B(\pi(x),\eta))$. Then $\pi(x) \in M_{1} \subseteq B(\pi(x),\eta)$, where $M_{1} = \cap_{m=1}^{s} \pi(I_{1}\cap S_{a_1}V_{m})$.

If $s=1$, by {\cite[Theorem C]{Bergelson_Leibman96}}, there exist infinitely many $n_1$ for which there exists $x_{1}\in I_{1}\cap S_{a_1}V_{1}$ such that, for each $(q_1,\dots,q_d)\in \cC$, $T_1^{p_{1,1}(n_1)}\cdots T_{d}^{p_{d,1}(n_1)}x_1,T_1^{q_1(n_1)}\cdots T_d^{q_d(n_1)}x_1\in I_{1}\cap S_{a_1}V_1$. In particular, $\pi(x_1),T_1^{q_1(n_1)}\cdots T_d^{q_d(n_1)}\pi(x_1)\in B(\pi(x),\eta)$. This proves the first step in the case $s=1$.

Assume now that $s\geq 2$. By \cref{lemma: reduce_type}, there exists $(p_1,\dots,p_d)\in \cP$ such that, for all large enough $h$, $(p_1,\dots,p_d,h)\vdc \cP$ has strictly smaller type than $\cP$. Suppose that $(p_1,\dots,p_d,h)\vdc \cP$ is very nice. Without loss of generality, we may assume that $(p_1,\dots,p_d)=(p_{1,1},\dots,p_{d,1})$. Let \begin{align*}
    \cC'=\big((-p_1,\dots,-p_d),(q_{1,1}-p_1,\dots,q_{d,1}-p_d),\dots,(q_{1,s'}-p_1,\dots,q_{d,s'}-p_d)\big).
\end{align*}

 Then $\pi$ has the property $\Lambda((p_1,\dots,p_d,h)\vdc \cP,\cC')$. Therefore, for the open sets $I_{1}\cap S_{a_{1}}V_{1}$ and  $\underbrace{X,\dots, X}_{s\text{ times}},I_{1}\cap S_{a_1}V_2,\dots,I_1 \cap S_{a_1} V_s$, there exist infinitely many $n_1\in \N$ for which there exists $y_1\in I_{1}\cap S_{a_1}V_1$ such that\begin{itemize}
    \item $T_1^{p_{1,m}(n_1)-p_1(n_1)}\cdots T_{d}^{p_{d,m}(n_1)-p_d(n_1)}y_1 \in I_{1}\cap S_{a_1}V_{m}$ for $2\leq m \leq s$.
    \item $T_1^{-p_1(n_1)}\cdots T_{d}^{-p_d(n_1)}\pi(y_1)$, $T_1^{q_1(n_1)-p_1(n_1)}\cdots T_d^{q_d(n_1)-p_d(n_1)}\pi(y_1)\in M_{1}$ for $(q_1,\dots,q_d)\in \cC$.
\end{itemize} 

Setting $x_{1} = T_1^{-p_1(n_1)}\cdots T_{d}^{-p_d(n_1)}y_1$, we conclude the desired result. 

If $(p_1,\dots,p_d,h)\vdc \cP$ is not very nice, then by \cref{lemma: reduce_type}, it can be reordered to form a very nice family. The same argument then applies after reordering the corresponding open sets in the same way.

\textbf{Step $\ell$:} Let $\ell\geq 2$ be an integer and assume that we have already chosen infinitely many $n_1,\dots,n_{\ell-1}$ for which there exist $x_1,\dots,x_{\ell-1}\in X$ such that\begin{itemize}
    \item $\pi(x_{\ell-1})\in B(\pi(x),(\ell-1)\eta)$;
    \item $T_1^{p_{1,m}(n_{j}+\dots+n_{\ell-1})}\cdots T_d^{p_{d,m}(n_{j}+\dots+n_{\ell-1})}x_{\ell-1}\in S_{a_{j}}V_{m}$ for $1\leq m\leq s$;
    \item $T_1^{q_1(n_{j}+\dots+n_{\ell-1})}\cdots T_d^{q_d(n_{j}+\dots+n_{\ell-1})} \pi(x_{\ell-1})\in B(\pi(x),(\ell-1)\eta)$ for $(q_1,\dots,q_d)\in \cC$.
\end{itemize}

Choose $0<\eta_{\ell}<\eta$ such that, for $1\leq j \leq \ell-1$,\begin{itemize}
    \item $T_1^{p_{1,m}(n_{j}+\dots+n_{\ell-1})}\cdots T_d^{p_{d,m}(n_{j}+\dots+n_{\ell-1})}B(x_{\ell-1},\eta_{\ell})\subseteq S_{a_{j}} V_{m}$ for every $1\leq m\leq s$;
    \item $T_1^{q_1(n_{j}+\dots+n_{\ell-1})}\cdots T_d^{q_d(n_{j}+\dots+n_{\ell-1})}B(\pi(x_{\ell-1}),\eta_{\ell}) \subseteq B(\pi(x),(\ell-1)\eta)$ for every $(q_1,\dots,q_d)\in \cC$.
\end{itemize}

Let $I_{\ell} = \pi^{-1}(B(\pi(x_{\ell-1}),\eta_{\ell}))$. Since $\pi(x_{\ell-1})\in B(\pi(x),\delta)\subseteq \cap_{j=1}^{N} S_{a_j}W_0$, $\pi(x_{\ell-1})\in M_{\ell} \subseteq \pi(I_{\ell})\subseteq B(\pi(x_{\ell-1}),\eta_\ell)$, where $M_{\ell} = \cap_{m=1}^{s}\pi(I_{\ell}\cap S_{a_{\ell}}V_{m}) \cap \pi(B(x_{\ell-1},\eta_\ell))$.

By \cref{lemma: reduce_type}, there exists $(p_1,\dots,p_d)\in \cP$ such that,  $(p_1,\dots,p_d,n_1,\dots,n_{\ell-1})\vdc \cP$ has strictly smaller type than $\cP$ for all large enough $n_1,\dots,n_{\ell-1}$. Without loss of generality, we may assume that $(p_1,\dots,p_d,n_1,\dots,n_{\ell-1})\vdc \cP$ is very nice and $(p_1,\dots,p_d)=(p_{1,1},\dots,p_{d,1})$.

Let $\cC'$ be the family consisting of the tuple $(-p_1,\dots,-p_d)$, the tuples $(q_{1,r}-p_1,\dots,q_{d,r}-p_d)$ for $1\leq r \leq s'$, and the tuples $(\partial_{n_j+\dots+n_{\ell-1}}q_{1,r}-p_1,\dots,\partial_{n_j+\dots+n_{\ell-1}}q_{d,r}-p_d)$ for $1\leq r \leq s'$ and $1\leq j\leq \ell-1$.

By the induction hypothesis, $\pi$ has the property $\Lambda((p_1,\dots,p_d,n_1,\dots,n_{\ell-1})\vdc \cP,\cC')$. Therefore, for open sets $I_{\ell}\cap S_{a_\ell}V_1,\dots,I_{\ell}\cap S_{a_\ell}V_s,\underbrace{B(x_{\ell-1},\eta_{\ell}),\dots, B(x_{\ell-1},\eta_{\ell})}_{s(\ell-1)\text{ times}}$, there exist infinitely many $n_\ell$ for which there exists $y_\ell \in I_{\ell} \cap S_{a_\ell} V_1$ such that, for $1\leq j \leq \ell-1$,\begin{enumerate}
    \item $T_1^{p_{1,m}(n_\ell)-p_1(n_\ell)}\cdots T_d^{p_{d,m}(n_\ell)-p_d(n_\ell)}y_\ell \in I_\ell \cap S_{a_\ell}V_m$ for $2\leq m\leq s$;
    \item $T_1^{\partial_{n_j + \dots+n_{\ell-1}}p_{1,m}(n_{\ell}) - p_1(n_\ell)}\cdots T_d^{\partial_{n_j + \dots+n_{\ell-1}}p_{d,m}(n_{\ell}) - p_d(n_\ell)}y_\ell \in B(x_{\ell-1},\eta_{\ell})$ for $1\leq m \leq s$;
    \item $T_1^{-p_1(n_{\ell})}\cdots T_d^{-p_d(n_{\ell})}\pi(y_\ell),T_1^{q_{1}(n_\ell)-p_1(n_\ell)}\cdots T_d^{q_{d}(n_\ell)-p_d(n_\ell)}\pi(y_\ell)\in M_{\ell}$ for $(q_1,\dots, q_d)\in \cC$;
    \item $T_1^{\partial_{n_{j}+\dots+n_{\ell-1}}q_1(n_\ell) -p_1(n_\ell)}\cdots T_d^{\partial_{n_{j}+\dots+n_{\ell-1}}q_d(n_\ell) -p_d(n_\ell)} \pi(y_\ell)\in M_{\ell} \subseteq B(\pi(x_{\ell-1}),\eta_\ell)$ for $(q_1,\dots,q_d)\in \cC$.
\end{enumerate}

Set $x_{\ell} = T_1^{-p_1(n_{\ell})}\cdots T_d^{-p_d(n_{\ell})}y_{\ell}$. By $(1)$, for $1\leq m\leq s$, we have $T_1^{p_{1,m}(n_\ell)}\cdots T_d^{p_{d,m}(n_\ell)}x_\ell \in S_{a_\ell}V_m$.

By $(2)$, for every $1\leq m\leq s$ and every $1\leq j\leq \ell-1$, we have\begin{align*}
&T_1^{p_{1,m}(n_j+\cdots+n_\ell)}\cdots
 T_d^{p_{d,m}(n_j+\cdots+n_\ell)}x_\ell       \\
&\quad =
T_1^{p_{1,m}(n_j+\cdots+n_{\ell-1})}
\cdots T_d^{p_{d,m}(n_j+\cdots+n_{\ell-1})} 
\Big(
T_1^{\partial_{n_j+\cdots+n_{\ell-1}}p_{1,m}(n_\ell)}\cdots
T_d^{\partial_{n_j+\cdots+n_{\ell-1}}p_{d,m}(n_\ell)}
x_\ell
\Big)                                        \\
&\quad \in
T_1^{p_{1,m}(n_j+\cdots+n_{\ell-1})}\cdots 
T_d^{p_{d,m}(n_j+\cdots+n_{\ell-1})}
B(x_{\ell-1},\eta_\ell)   \subseteq S_{a_j}V_m.
\end{align*}

Similarly, by $(4)$, for every $(q_1,\dots,q_d)\in\cC$ and every $1\leq j\leq \ell-1$, we obtain\begin{align*}
    T_1^{q_1(n_j+\cdots+n_\ell)}\cdots T_d^{q_d(n_j+\cdots+n_\ell)}\pi(x_\ell) \in T_1^{q_1(n_j+\cdots+n_{\ell-1})}\cdots T_d^{q_d(n_j+\cdots+n_{\ell-1})}B(\pi(x_{\ell-1}),\eta_\ell) \subseteq B(\pi(x),\ell\eta).
\end{align*}

Finally, by $(3)$, for every $(q_1,\dots,q_d)\in \cC$, we have $\pi(x_\ell),T_1^{q_1(n_\ell)}\cdots T_d^{q_d(n_\ell)}\pi(x_\ell)\in M_\ell \subseteq B(\pi(x),\ell\eta)$.

\subsubsection*{Case 2: there exists $m\in \{1,\dots,s\}$ such that $\deg((p_{1,m},\dots,p_{d,m}))=1$.}

We need two intermediate claims, whose proofs will be given later.

\begin{claim}\label{claim 1}
     Assume that $\pi$ has the property $\Lambda(\cP',\cC')$ for any very nice family of $d$-tuples of polynomials $\cP'$ with $\deg((p_1,\dots,p_d))\geq 2$ for every $(p_1,\dots,p_d)\in \cP'$ and matrix type $W$, and any family of $d$-tuples of polynomials $\cC'$. Let \begin{align*}
        \cP=(\cP_1,\dots,\cP_d) &= \big((p_{1,1},\dots,p_{d,1}),\dots,(p_{1,s},\dots,p_{d,s})\big)
    \end{align*}

     be a very nice family of $d$-tuples of polynomials with $\deg((p_{1,m},\dots,p_{d,m}))\geq 2$ for every $1\leq m\leq s$ and matrix type $W$, and let $\cB$ be a family of $d$-tuples of polynomials such that $\deg((b_1,\dots,b_d))<2$ for all $(b_1,\dots,b_d)\in\cB$. Suppose that $\cP\cup\cB$ is very nice. Then for any family of $d$-tuples of polynomials $\cC$ and open subsets $U_0,U_1,\dots,U_s$ of $X$ with $\cap_{m=0}^{s}\pi(U_m)\neq\emptyset$, there exist infinitely many $n\in \N$ for which there exists $z\in U_0$ such that\begin{itemize}
        \item $T_1^{b_1(n)}\cdots T_d^{b_d(n)}z\in U_0$ for $(b_1,\dots,b_d)\in \cB$;
        \item $T_1^{p_{1,m}(n)}\cdots T_d^{p_{d,m}(n)}z\in U_m$ for $1\leq m \leq s$;
        \item $T_1^{q_1(n)}\cdots T_d^{q_d(n)}\pi(z)\in \cap_{m=0}^{s}\pi(U_m)$ for $(q_1,\dots,q_d)\in \cC$.
    \end{itemize}
\end{claim}

\begin{claim}\label{claim 2}
    Assume that $\pi$ has the property $\Lambda(\cP',\cC')$ for any very nice family of $d$-tuples of polynomials $\cP'$ with $\deg((p_1,\dots,p_d))\geq 2$ for every $(p_1,\dots,p_d)\in \cP'$ and matrix type $W$ and any family of $d$-tuples of polynomials $\cC'$. Let $\cP$ be a very nice family of $d$-tuples of polynomials with $\deg((p_1,\dots,p_d))\geq 2$ for every $(p_1,\dots,p_d)\in \cP$ and matrix type $W$ and let \begin{align*}
        \cB= \big((b_{1,1},\dots,b_{d,1}),\dots,(b_{1,t},\dots,b_{d,t})\big)
    \end{align*}

    be a family of $d$-tuples of polynomials such that $\deg(\cB)=1$. Suppose that $\pi$ has the property $\Lambda(\cB,\mathcal{D})$ for any family of $d$-tuples of polynomials $\mathcal{D}$ and that $\cP\cup\cB$ is very nice. Then for any family of $d$-tuples of polynomials $\cC$ and open subsets $U_0,U_1,\dots,U_t$ of $X$ with $\cap_{i=0}^{t}\pi(U_i)\neq\emptyset$, there exist infinitely many $n\in \N$ for which there exists $z\in U_0$ such that\begin{itemize}
        \item $T_1^{p_1(n)}\cdots T_d^{p_d(n)}z\in U_0$ for $(p_1,\dots,p_d)\in \cP$;
        \item $T_1^{b_{1,i}(n)}\cdots T_d^{b_{d,i}(n)}z\in U_i$ for $1\leq i\leq t$;
        \item $T_1^{q_1(n)}\cdots T_d^{q_d(n)}\pi(z)\in \cap_{i=0}^{t}\pi(U_i)$ for $(q_1,\dots,q_d)\in \cC$.
    \end{itemize}\end{claim}

    Let $t$ be the smallest integer such that $\deg((p_{1,t+1},\dots,p_{d,t+1}))=1$. Since $\cP$ is very nice, it follows that $\deg((p_{1,i},\dots,p_{d,i}))=1$ for $t+1\leq i \leq s$.

    Let \begin{align*}
    \cB &= \big((p_{1,t+1},\dots,p_{d,t+1}),\dots,(p_{1,s},\dots,p_{d,s})\big),\\
    \cP' &= \big((p_{1,1},\dots,p_{d,1}),\dots,(p_{1,t},\dots,p_{d,t})\big).
    \end{align*}

Then $\cP'$ is a very nice family of type at most $W$. Since $\deg((p_{1,m},\dots,p_{d,m}))\geq 2$ for every $1\leq m\leq t$, Case $1$ implies that $\pi$ has the property $\Lambda(\cP',\cC)$. Moreover, since $\cP$ is very nice, if $p_{1,t+1}\neq 0$ then $\cB$ is nice, and consequently $p_{j,i}=0$ for every $t+1\leq i\leq s$ and $j\geq 2$. If $p_{1,t+1}=0$, let $j$ be the smallest integer such that $p_{j,t+1}\neq 0$. Since $\cP$ is very nice, we have $p_{j',i}= 0$ for every $j'\in \{1,\dots,j-1,j+1,\dots, d\}$ and $t+1\leq i \leq s$. Thus, in either case, there exists $j\in\{1,\dots,d\}$ such that $p_{j',i}=0$ for every $t+1\leq i\leq s$ and every $1\leq j'\leq d$ with $j'\neq j$. Thus, by \cref{lemma: linear stronger}, it follows that $\pi$ has the property $\Lambda(\cB,\mathcal{D})$ for any family of $d$-tuples of polynomials $\mathcal{D}$.

Applying \cref{claim 1} to $\cP'$, $\cB$, $\cC$ and the open sets $V_{0},V_{1},\dots,V_{t}$, there exist infinitely many $k\in \N$ for which there exists $x\in V_0$ such that \begin{itemize}
    \item $T_{1}^{b_1(k)}\cdots T_d^{b_d(k)}x \in V_0$ for $(b_1,\dots,b_d)\in \cB$;
    \item $T_{1}^{p_{1,m}(k)}\cdots T_{d}^{p_{d,m}(k)}x\in V_m$ for $1\leq m \leq t$;
    \item $T_{1}^{q_1(k)}\cdots T_{d}^{q_d(k)}\pi(x)\in W_0$ for $(q_1,\dots,q_d)\in \cC$.
\end{itemize}

Thus, for $1\leq i\leq s-t$, there is some $x_i\in V_{i+t}$ such that $\pi(x) =\pi (T_{1}^{-p_{1,i+t}(k)}\cdots T_{d}^{-p_{d,i+t}(k)}x_i)$. Choose $\eta>0$ such that $T_{1}^{p_{1,i+t}(k)}\cdots T_{d}^{p_{d,i+t}(k)}B(T_{1}^{-p_{1,i+t}(k)}\cdots T_{d}^{-p_{d,i+t}(k)}x_i,\eta)\subseteq V_{i+t}$ for $1\leq i \leq s-t$, $T_{1}^{p_{1,m}(k)}\cdots T_{d}^{p_{d,m}(k)}B(x,\eta)\subseteq V_m$ for $1\leq m \leq t$ and $T_{1}^{q_1(k)}\cdots T_{d}^{q_d(k)}B(\pi(x),\eta)\subseteq W_0$ for $(q_1,\dots,q_d)\in\cC$.

Let $W_0' = B(x,\eta)\cap \pi^{-1}(B(\pi(x),\eta))\cap V_0$ and let $W_i = B(T_{1}^{-p_{1,i+t}(k)}\cdots T_{d}^{-p_{d,i+t}(k)}x_i,\eta)$ for $1\leq i \leq s-t$. Notice that $W_0'$ is a nonempty open set as $x\in W_0'$ and $\pi(x)\in \pi(W_i)$ for $1\leq i \leq s-t$. Let\begin{align*}
    \cP'' &= \big((\partial_k p_{1,1},\dots,\partial_k p_{d,1}),\dots,(\partial_k p_{1,t},\dots,\partial_k p_{d,t})\big),\\
    \cC' &= \big((\partial_k q_{1,1},\dots,\partial_k q_{d,1}),\dots,(\partial_k q_{1,s'},\dots,\partial_k q_{d,s'})\big).
\end{align*}

Since $\cP''$ has type at most $W$ and $\deg((p_1,\dots,p_d))\geq 2$ for every $(p_1,\dots,p_d)\in \cP''$, it follows from Case $1$ that $\pi$ has the property $\Lambda(\cP'',\cC')$. Applying \cref{claim 2} to $\cP''$, $\cB$, $\cC'$ and the open sets $W_0',W_1,\dots,W_{s-t}$, there exist infinitely many $n\in \N$  for which there exists $z\in W_0'$ such that \begin{enumerate}
    \item $T_{1}^{\partial_k p_{1,m}(n)}\cdots T_{d}^{\partial_k p_{d,m}(n)}z\in W_0'\subseteq B(x,\eta)$ for $1\leq m \leq t$;
    \item $T_{1}^{p_{1,i+t}(n)}\cdots T_{d}^{p_{d,i+t}(n)}z\in W_i$ for $1\leq i \leq s-t$;
    \item $T_{1}^{\partial_k q_1(n)}\cdots T_{d}^{\partial_k q_d(n)}\pi(z) \in \pi(W_0')\cap  \cap_{i=1}^{s-t}\pi(W_i)\subseteq B(\pi(x),\eta)$ for $(q_1,\dots,q_d)\in \cC$.
\end{enumerate}

By $(1)$, for every $1\leq m\leq t$, we have\begin{align*}
T_1^{p_{1,m}(n+k)}\cdots 
 T_d^{p_{d,m}(n+k)}z
&\quad =
T_1^{p_{1,m}(k)}\cdots 
T_d^{p_{d,m}(k)} 
\Big(
T_1^{\partial_{k}p_{1,m}(n)}\cdots 
T_d^{\partial_{k}p_{d,m}(n)}
z
\Big)                                        \\
&\quad \in
T_1^{p_{1,m}(k)}\cdots
T_d^{p_{d,m}(k)}
B(x,\eta)   \subseteq V_m.
\end{align*}

Similarly, by $(3)$, for every $(q_1,\dots,q_d)\in\cC$, we obtain\begin{align*}
    T_1^{q_1(n+k)}\cdots  T_d^{q_d(n+k)}\pi(z) \in T_1^{q_1(k)} \cdots T_d^{q_d(k)}B(\pi(x),\eta) \subseteq W_0.
\end{align*}

Finally, by $(2)$, for every $1\leq i \leq s-t$, we have\begin{align*}
T_1^{p_{1,i+t}(n+k)}\cdots
 T_d^{p_{d,i+t}(n+k)}z
&\quad =
T_1^{p_{1,i+t}(k)}\cdots
T_d^{p_{d,i+t}(k)} 
\big(T_1^{p_{1,i+t}(n)}\cdots T_d^{p_{d,i+t}(n)}z\big)                                        \\
&\quad \in
T_1^{p_{1,i+t}(k)}\cdots
T_d^{p_{d,i+t}(k)}
W_i   \subseteq V_{i+t}.
\end{align*}

It remains to prove the claims.

\begin{proof}[Proof of \cref{claim 1}]
    The claim is clear if $\cB$ is empty. Hence assume that $\cB\neq\emptyset$. We prove the claim by induction on the type of $\cB$.
    
    Fix a nonempty family of $d$-tuples of polynomials $\cB = (\cB_1,\dots,\cB_d)$ such that $\deg((b_1,\dots,b_d))<2$ for all $(b_1,\dots,b_d)\in \cB$ and $\cP\cup\cB$ is very nice. Assume that the claim has already been proved for all families of $d$-tuples of polynomials whose type is strictly smaller than that of $\cB$.

    Let $W_0 = \cap_{m=0}^{s}\pi(U_m)$. As before, we may assume that $\pi(U_m)=W_0$ for $0\leq m \leq s$, and there exists $x\in X$ with $\pi(x)\in W_0$, $a_1,\dots,a_N\in \Z^{k}$ and $\delta>0$ such that $\pi^{-1}(B(\pi(x),\delta))\subseteq\cap_{m=0}^{s}\cup_{j=1}^{N}S_{a_j}U_m$, and $B(\pi(x),\delta)\subseteq W_0 \cap \left(\cap_{j=1}^{N}S_{a_j}W_0 \right)$.

    Let $\eta=\delta/N$. Inductively we will construct infinitely many $n_1,\dots,n_N$ for which there exist $x_1,\dots,x_N\in X$ and $k_1,\dots,k_{N+1}\in \{1,\dots,N\}$ with $k_1=1$, such that for $1\leq j \leq \ell\leq N$,\begin{itemize}
        \item $x_\ell\in S_{a_{k_{\ell+1}}}U_0$;
        \item $\pi(x_\ell)\in B(\pi(x),\ell\eta)$;
        \item $T_1^{b_1(n_j+\dots+n_\ell)}\cdots T_d^{b_d(n_j+\dots+n_\ell)}x_\ell\in S_{a_{k_j}}U_0$ for $(b_1,\dots,b_d)\in \cB$;
        \item $T_1^{p_{1,m}(n_j+\dots+n_\ell)}\cdots T_d^{p_{d,m}(n_j+\dots+n_\ell)}x_\ell \in S_{a_{k_j}}U_m$ for $1\leq m\leq s$;
        \item $T_1^{q_1(n_j+\dots+n_\ell)}\cdots T_d^{q_d(n_j+\dots+n_\ell)}\pi(x_\ell)\in B(\pi(x),\ell\eta)$ for $(q_1,\dots,q_d)\in \cC$.
    \end{itemize} 

    Note that if this has been achieved, choose $1\leq j \leq \ell\leq N$ with $k_j = k_{\ell+1}$. Then, setting $z = S_{-a_{k_j}}x_{\ell}$, we conclude the desired result for $n=n_j+\dots+n_\ell$.

    \textbf{Step 1:} Let $I_1 = \pi^{-1}(B(\pi(x),\eta))$. Then $\pi(x)\in M_1\subseteq \pi(I_1)$, where $M_1 = \cap_{m=0}^{s}\pi(I_1\cap S_{a_1}U_m)$.

    Recall that, for $i=1,\dots, d$,\begin{align*}
        \widehat{\cB}_i=\{b_i\colon (b_1,\dots,b_d)\in \cB, b_i \text{ is nonconstant and }  b_{i'} \text{ is constant for } i'<i\}.
    \end{align*}

     If $\widehat{\cB}_i\neq\emptyset$ for some $i\geq 2$, let $i$ be maximal with this property and choose $(b_1,\dots,b_d)$ such that $b_1=\cdots=b_{i-1}=0$ and $b_i\in \widehat{\cB}_i$. Otherwise, choose any $(b_1,\dots,b_d)\in \cB$.
        
    Let \begin{align*}
        \cB' &= (\cB_1 -b_1,\dots,\cB_d-b_d)^*,\\
        \cP' &= (\cP_1 - b_1,\dots,\cP_d-b_d),\\
        \cC' &= \big((-b_1,\dots,-b_d),(q_{1,1}-b_1,\dots,q_{d,1}-b_d),\dots, (q_{1,s'}-b_1,\dots,q_{d,s'}-b_d)\big ).
    \end{align*}

    As in the proof of \cref{lemma: reduce_type}, the family $\cB'$ has strictly smaller type than $\cB$. 
    
    Let $(p_1,\dots,p_d)\in \cP$. Suppose first that $p_1\neq 0$. Since $\cP$ is very nice and $\deg((p_1,\dots,p_d))\geq 2$, we have $\deg(p_1)\geq 2$, and hence $p_1-b_1\sim p_1$. If $p_1=0$, let $i$ be the smallest integer such that $p_i\neq 0$. Since $\cP\cup\cB$ is very nice, we have $b_{i'}=0$ for $1\leq i'< i$. As before, since $\deg((p_1,\dots,p_d))\geq 2$ and $\deg((b_1,\dots,b_d))= 1$, $\deg(p_i)\geq 2$ and $p_i-b_i\sim p_i$. Hence, $\cP'$ is a very nice family with type $W$ and $\deg((p_1',\dots,p_d'))\geq 2$ for every $(p_1',\dots,p_d')\in \cP'$. 

    By the induction hypothesis, for $\cC'$ and open sets $I_1\cap S_{a_1}U_0, I_1\cap S_{a_1}U_1,\dots,I_1\cap S_{a_1}U_s$ there exist infinitely many $n_1\in \N$ for which there exists $y_1 \in I_{1}\cap S_{a_1}U_0$ such that\begin{itemize}
        \item $T_1^{b_1'(n_1)-b_1(n_1)}\cdots T_d^{b_d'(n_1)-b_d(n_1)}y_1\in I_1\cap S_{a_1}U_0$ for $(b_1',\dots,b_d')\in \cB$;
        \item $T_1^{p_{1,m}(n_1)-b_1(n_1)}\cdots T_d^{p_{d,m}(n_1)-b_d(n_1)}y_1\in I_1\cap S_{a_1}U_m$ for $1\leq m \leq s$;
        \item $T_1^{-b_1(n_1)}\cdots T_d^{-b_d(n_1)}\pi(y_1), T_1^{q_1(n_1)-b_1(n_1)}\cdots T_d^{q_d(n_1)-b_d(n_1)}\pi(y_1)\in M_1$ for $(q_1,\dots,q_d)\in \cC$.
    \end{itemize}

    Setting $x_1 = T_1^{-b_1(n_1)}\cdots T_d^{-b_d(n_1)}y_1$ and choosing $k_2\in \{1,\dots,N\}$ such that $x_1\in S_{a_{k_2}}U_0$, we obtain the desired conclusion.

    \textbf{Step $\ell$:} Let $\ell\geq 2$ be an integer and assume that we have already chosen  infinitely many $n_1,\dots,n_{\ell-1}$, $x_1,\dots,x_{\ell-1}\in X$ and $k_1,\dots,k_{\ell}\in \{1,\dots,N\}$ with the required properties up to level $\ell-1$.

    Choose $0<\eta_\ell<\eta$ such that, for $1\leq j \leq \ell-1$,\begin{itemize}
        \item $B(x_{\ell-1},\eta_{\ell})\subseteq S_{a_{k_\ell}}U_0$;
        \item $T_1^{b_1(n_j+\dots+n_{\ell-1})}\cdots T_d^{b_d(n_j+\dots+n_{\ell-1})}B(x_{\ell-1},\eta_\ell)\subseteq S_{a_{k_j}}U_0$ for $(b_1,\dots,b_d)\in \cB$;
        \item $T_1^{p_{1,m}(n_j+\dots+n_{\ell-1})}\cdots T_d^{p_{d,m}(n_j+\dots+n_{\ell-1})}B(x_{\ell-1},\eta_\ell)\subseteq S_{a_{k_j}}U_m$ for $1\leq m \leq s$;
        \item $T_1^{q_1(n_j+\dots+n_{\ell-1})}\cdots T_d^{q_d(n_j+\dots+n_{\ell-1})}B(\pi(x_{\ell-1}),\eta_\ell)\subseteq B(\pi(x),(\ell-1)\eta)$ for $(q_1,\dots,q_d)\in \cC$.
    \end{itemize}

    Let $I_{\ell} = \pi^{-1}(B(\pi(x_{\ell-1}),\eta_\ell))$. Note that $\pi(x_{\ell-1})\in M_{\ell}\subseteq \pi(I_\ell)$, where $M_{\ell} = \cap_{m=1}^{s} \pi(I_\ell \cap S_{a_{k_\ell}}U_m) \cap \pi(B(x_{\ell-1},\eta_{\ell}))$.

    If $\widehat{\cB}_i\neq\emptyset$ for some $i\geq 2$, let $i$ be maximal with this property and choose $(b_1,\dots,b_d)$ such that $b_1=\cdots=b_{i-1}=0$ and $b_i\in \widehat{\cB}_i$. Otherwise, choose any $(b_1,\dots,b_d)\in \cB$.
    
    Let $\cB' = (\cB_1 -b_1,\dots,\cB_d-b_d)^*$. This family has strictly smaller type than $\cB$. Let $\cC'$ be the family consisting of $(-b_1,\dots,-b_d)$, the tuples $(q_1-b_1,\dots,q_d-b_d)$ for $(q_1,\dots,q_d)\in \cC$, and the tuples $(\partial_{n_j+\dots+n_{\ell-1}}q_1-b_1,\dots,\partial_{n_j+\dots+n_{\ell-1}}q_d-b_d)$ for $(q_1,\dots,q_d)\in \cC$ and $1\leq j \leq \ell-1$. 

    As before, we may assume without loss of generality that $(b_1,\dots,b_d,n_1,\dots,n_{\ell-1})\vdc\cP$ is very nice. By an argument similar to that given in Step $1$, this family has matrix type $W$.
    
    By the induction hypothesis, for $\cC'$ and open sets $\underbrace{B(x_{\ell-1},\eta_{\ell}),\dots, B(x_{\ell-1},\eta_{\ell})}_{s(\ell-1)\text{ times}},B(x_{\ell-1},\eta_{\ell}),I_{\ell}\cap S_{a_{k_{\ell}}}U_1,\dots,I_{\ell}\cap S_{a_{k_{\ell}}}U_s$, there exist infinitely many $n_\ell$ for which there exists $y_\ell\in B(x_{\ell-1},\eta_{\ell})$ such that for $1\leq j \leq \ell-1$,\begin{enumerate}
        \item $T_{1}^{b_1'(n_\ell)-b_1(n_\ell)}\cdots T_{d}^{b_d'(n_\ell)-b_d(n_\ell)}y_\ell \in B(x_{\ell-1},\eta_{\ell})$ for $(b_1',\dots,b_d')\in \cB$;
        \item $T_{1}^{p_{1,m}(n_\ell)-b_1(n_\ell)}\cdots T_{d}^{p_{d,m}(n_\ell)-b_d(n_\ell)}y_\ell \in I_\ell \cap S_{a_{k_\ell}}U_m$ for $1\leq m \leq s$;
        \item $T_1^{\partial_{n_j+\dots+n_{\ell-1}}p_{1,m}(n_\ell)-b_1(n_\ell)} \cdots T_d^{\partial_{n_j+\dots+n_{\ell-1}}p_{d,m}(n_\ell)-b_d(n_\ell)}y_\ell \in B(x_{\ell-1},\eta_\ell)$ for $1\leq m\leq s$;
        \item $T_1^{-b_1(n_\ell)}\cdots T_d^{-b_d(n_\ell)}\pi(y_\ell)$, $T_{1}^{q_1(n_\ell)-b_1(n_\ell)}\cdots T_{d}^{q_d(n_\ell)-b_d(n_\ell)}\pi(y_\ell)\in M_{\ell}$ for $(q_1,\dots,q_d)\in \cC$;
        \item $T_1^{\partial_{n_j+\dots+n_{\ell-1}}q_1(n_\ell)-b_1(n_\ell)}\cdots T_d^{\partial_{n_j+\dots+n_{\ell-1}}q_d(n_\ell)-b_d(n_\ell)}\pi(y_\ell)\in M_{\ell}$ for $(q_1,\dots,q_d)\in \cC$.
    \end{enumerate}

    Set $x_{\ell} = T_{1}^{-b_1(n_\ell)}\cdots T_d^{-b_d(n_\ell)}y_{\ell}$ and choose $k_{\ell+1}\in \{1,\dots,N\}$ such that $x_\ell\in S_{a_{k_{\ell+1}}}U_0$. By $(2)$, for $1\leq m\leq s$, we have $T_{1}^{p_{1,m}(n_\ell)}\cdots T_d^{p_{d,m}(n_\ell)}x_\ell\in S_{a_{k_\ell}}U_m$.

    By $(3)$, for every $1\leq m\leq s$ and every $1\leq j \leq \ell-1$, we have\begin{align*}
&T_1^{p_{1,m}(n_j+\cdots+n_\ell)}\cdots 
 T_d^{p_{d,m}(n_j+\cdots+n_\ell)}x_\ell       \\
&\quad =
T_1^{p_{1,m}(n_j+\cdots+n_{\ell-1})}\cdots 
T_d^{p_{d,m}(n_j+\cdots+n_{\ell-1})} 
\Big(
T_1^{\partial_{n_j+\cdots+n_{\ell-1}}p_{1,m}(n_\ell)}\cdots 
T_d^{\partial_{n_j+\cdots+n_{\ell-1}}p_{d,m}(n_\ell)}
x_\ell
\Big)                                        \\
&\quad \in
T_1^{p_{1,m}(n_j+\cdots+n_{\ell-1})}\cdots
T_d^{p_{d,m}(n_j+\cdots+n_{\ell-1})}
B(x_{\ell-1},\eta_\ell)   \subseteq S_{a_{k_j}}U_m .
\end{align*}

Similarly, by $(5)$, for every $(q_1,\dots,q_d)\in\cC$ and every $1\leq j\leq \ell-1$, we obtain\begin{align*}
    T_1^{q_1(n_j+\cdots+n_\ell)}\cdots  T_d^{q_d(n_j+\cdots+n_\ell)}\pi(x_\ell) \in T_1^{q_1(n_j+\cdots+n_{\ell-1})} \cdots T_d^{q_d(n_j+\cdots+n_{\ell-1})}B(\pi(x_{\ell-1}),\eta_\ell) \subseteq B(\pi(x),\ell\eta).
\end{align*}

By $(4)$, for every $(q_1,\dots,q_d)\in \cC$, we have $\pi(x_\ell),T_1^{q_1(n_\ell)}\cdots T_d^{q_d(n_\ell)}\pi(x_\ell)\in M_\ell \subseteq B(\pi(x),\ell\eta)$.

Finally, by $(1)$, for every $(b_1',\dots,b_d')\in \cB$ and
$1\leq j \leq \ell-1$, we have\begin{align*}
&T_1^{b_{1}'(n_j+\cdots+n_\ell)}\cdots
 T_d^{b_{d}'(n_j+\cdots+n_\ell)}x_\ell       \\
&\quad =
T_1^{b_{1}'(n_j+\cdots+n_{\ell-1})}\cdots
T_d^{b_{d}'(n_j+\cdots+n_{\ell-1})} 
\big(T_1^{b_1'(n_\ell)}\cdots T_d^{b_d'(n_\ell)}x_\ell\big)                                        \\
&\quad \in
T_1^{b_1'(n_j+\cdots+n_{\ell-1})}\cdots
T_d^{b_d'(n_j+\cdots+n_{\ell-1})}
B(x_{\ell-1},\eta_\ell)   \subseteq S_{a_{k_j}}U_0 .
\end{align*}

\end{proof}

\begin{proof}[Proof of \cref{claim 2}]
    Let $W_0 = \cap_{i=0}^{t}\pi(U_i)$. We may assume that $\pi(U_i)=W_0$ for $0\leq i \leq t$. Since $\pi$ has the property $\Lambda(\cB,\cC \cup \cP)$, for the open sets $U_0,U_1,\dots,U_t$ there exist infinitely many $k\in \N$ for which there exists $y\in U_0$ such that\begin{itemize}
        \item $T_1^{b_{1,i}(k)}\cdots T_d^{b_{d,i}(k)}y\in U_i$ for $1\leq i \leq t$;
        \item $T_1^{p_1(k)}\cdots T_d^{p_d(k)}\pi(y),T_1^{q_1(k)}\cdots T_d^{q_d(k)}\pi(y)\in W_0$ for $(p_1,\dots,p_d)\in \cP$ and $(q_1,\dots,q_d)\in \cC$.
    \end{itemize}

    Fix a sufficiently large $k$ with this property, and let\begin{align*}
        V_0 = U_0 \cap \bigcap_{i=1}^{t}T_1^{-b_{1,i}(k)}\cdots T_d^{-b_{d,i}(k)}U_i.
    \end{align*}
    
    This set is nonempty since $y\in V_0$. For each $(p_1,\dots,p_d)\in \cP$, let\begin{align*}
        V_{(p_1,\dots,p_d)} = T_1^{-p_1(k)}\cdots T_d^{-p_d(k)}U_0.
    \end{align*}
    
    Define\begin{align*}
        M= W_0 \cap \bigcap_{(q_1,\dots,q_d)\in \cC}T_1^{-q_1(k)}\cdots T_d^{-q_d(k)}W_0 \cap \pi(V_0) \cap \bigcap_{(p_1,\dots,p_d)\in \cP} \pi(V_{(p_1,\dots,p_d)}).
    \end{align*}

    The set $M$ is nonempty since $\pi(y)\in M$. By replacing $V_0$ and each $V_{(p_1,\dots,p_d)}$ with their intersections with $\pi^{-1}(M)$, we may assume that $M=\pi(V_0)=\pi(V_{(p_1,\dots,p_d)})$ for every $(p_1,\dots,p_d)\in \cP$.

    Let $\cP'$ be the very nice family consisting of the tuples $(\partial_k p_1,\dots,\partial_k p_d)$ with $(p_1,\dots,p_d)\in \cP$, and let $\cC'$ be the family consisting of the tuples $(\partial_k q_1,\dots,\partial_k q_d)$ with $(q_1,\dots,q_d)\in \cC$. Then $\pi$ has the property $\Lambda(\cP',\cC')$.
    
    Therefore, by \cref{claim 1}, for $\cP'$, $\cB$, $\cC'$, and the open sets $V_0,V_{(p_1,\dots,p_d)}$ for $(p_1,\dots,p_d)\in \cP$, there exist infinitely many $n\in\N$ for which there exists $z\in V_0$ such that
    
    \begin{itemize}
        \item $T_1^{b_1(n)}\cdots T_d^{b_d(n)}z\in V_0$ for $(b_1,\dots,b_d)\in \cB$;
        \item $T_1^{\partial_k p_1(n)}\cdots T_d^{\partial_k p_d(n)}z \in V_{(p_1,\dots,p_d)}$ for $(p_1,\dots,p_d)\in \cP$;
        \item $T_1^{\partial_k q_1(n)}\cdots T_d^{\partial_k q_d(n)}\pi(z)\in M$ for $(q_1,\dots,q_d)\in \cC$.
    \end{itemize}

    Since $V_0\subseteq U_0$, we have $z\in U_0$. Moreover, for $1\leq i \leq t$, since $V_0 \subseteq T_1^{-b_{1,i}(k)}\cdots T_d^{-b_{d,i}(k)}U_i$, it follows that\begin{align*}
        T_1^{b_{1,i}(n+k)}\cdots T_d^{b_{d,i}(n+k)}z \in T_1^{b_{1,i}(k)}\cdots T_d^{b_{d,i}(k)}V_0 \subseteq U_i.
    \end{align*}

    For each $(p_1,\dots,p_d)\in \cP$, by the definition of the sets $V_{(p_1,\dots,p_d)}$, we have\begin{align*}
        T_1^{p_1(n+k)}\cdots T_d^{p_d(n+k)}z = T_{1}^{p_1(k)+\partial_k p_1(n)}\cdots T_{d}^{p_d(k)+\partial_k p_d(n)}z \in T_1^{p_1(k)}\cdots T_d^{p_d(k)}V_{(p_1,\dots,p_d)} \subseteq U_0.
    \end{align*}

    Similarly, by the definition of $M$, for each $(q_1,\dots,q_d)\in\cC$ we have\begin{align*}
        T_1^{q_1(n+k)}\cdots T_d^{q_d(n+k)}\pi(z) \in T_1^{q_1(k)}\cdots T_d^{q_d(k)}M \subseteq W_0.
    \end{align*}
\end{proof}

\end{proof}

\section{Topological characteristic factors for distinct-degree polynomials}

In this section, we introduce the topological rational Kronecker factor of a system and prove \cref{thm: TFC_distinct_deg}.

\subsection{The topological rational Kronecker factor}\label{subsec: Krat}

Let $(X,T)$ be a system, and let $\mu$ be an invariant probability measure on $X$. Let $\mathcal{K}_{\operatorname{rat}}(T)$ be the closed subspace of $L^{2}(\mu)$ spanned by all $f\in L^{2}(\mu)$ such that $f\neq 0$ and $Tf = e(\alpha) f$ for some $\alpha\in\Q$, where $(Tf)(x)=f(Tx)$ and $e(\alpha)=e^{2\pi i \alpha}$. We denote by $Z_{\operatorname{rat}}$ the measurable factor induced by $\mathcal{K}_{\operatorname{rat}}(T)$.

We now introduce the topological analogue of  $Z_{\operatorname{rat}}$. Define\begin{align*}
    \mathcal{E}_{\operatorname{rat}}(X,T) = \{f\in C(X): Tf = e(\alpha) f \text{ for some }\alpha\in\Q\},
\end{align*}

and \begin{align*}
    \Krat(X,T) = \{(x,y)\in X\times X: f(x)=f(y)\text{ for all }f\in \mathcal{E}_{\operatorname{rat}}(X,T)\}.
\end{align*}

The relation $\Krat(X,T)$ is a closed invariant equivalence relation. We call the quotient system $(X/\Krat(X,T),T)$ the \emph{topological rational Kronecker factor} of $(X,T)$ and denote it by $X_{\operatorname{rat}}$. In general, $X_{\operatorname{rat}}$ and $Z_{\operatorname{rat}}$ need not coincide, see for instance, \cite[Section 7]{Bressaud_Durand_Maass_Eigenvalues_finite_rank:2010}. However, if $X$ is a pro-nilsystem, then $X_{\operatorname{rat}}$ and $Z_{\operatorname{rat}}$ coincide. 

Our definition of the topological rational Kronecker factor differs from the one introduced in \cite{Glasscock_Koutsogiannis_Richter_mult_comb_return:2019} for minimal systems. We will show that the two definitions coincide when the system is minimal. We first prove the following properties of the topological rational Kronecker factor. Here, an \emph{adding machine} is a system isomorphic to an inverse limit of finite cyclic systems $(\Z/r_i\Z,n\mapsto n+1)$, where $r_i\mid r_{i+1}$ for every $i\in\N$.

\begin{proposition}
    Let $(X,T)$ be a system. Then $(X_{\operatorname{rat}},T)$ is an equicontinuous system. Furthermore, if $(X,T)$ is transitive then $(X_{\operatorname{rat}},T)$ is isomorphic to an adding machine.
\end{proposition}

\begin{proof}
    Let $(x,y)\in \RP^{[1]}(X)$ and let $f\in\mathcal{E}_{\operatorname{rat}}(X,T)$. Since $(x,y)\in \RP^{[1]}(X)$, there exist $(n_i)_{i\in\N} \subseteq \Z$ and $(x_i)_{i\in\N},(y_{i})_{i\in \N}$ such that\begin{align*}
        (x_i,y_i)\to (x,y) \quad \text{ and } (T^{n_i}x_i,T^{n_i}y_i)\to (a,a),
    \end{align*}

    for some $a\in X$.

    Let $\alpha\in\Q$ such that $Tf = e(\alpha) f$. Without loss of generality, we may assume that $e(n_i\alpha)\to c$ for some $c\in \C$. Note that $c\neq 0$ since $|e(n_i\alpha)|=1$.
    
    Since $f$ is continuous, we have \begin{align*}
        (T^{n_i}f(x_i),T^{n_i}f(y_i))= (e(n_i\alpha)f(x_i),e(n_i\alpha)f(y_i))\to (cf(x),cf(y)) =(f(a),f(a)).
    \end{align*}

    Hence, $f(x)=f(y)$. Therefore, $(x,y)\in \Krat(X,T)$. Since the maximal equicontinuous factor is characterized as the maximal factor with $\RP^{[1]}=\Delta$ (see, for instance, {\cite[Chapter 9]{Auslander_minimal_flows_and_extensions:1988}}), it follows that $(X_{\operatorname{rat}},T)$ is equicontinuous.

    % Let $\pi: X\to X_{\operatorname{rat}}$ be the quotient map. For each $f\in\mathcal{E}_{\operatorname{rat}}(X,T)$, choose $\alpha_f\in\Q$ such that $Tf =e(\alpha_f)f$. Let \begin{align*}
    %     Z= \prod_{f\in \mathcal{E}_{\operatorname{rat}}(X,T)}\overline{f(X)}.
    % \end{align*}

    % Define $S: Z\to Z$ by $S((y_f)_{f\in \mathcal{E}_{\operatorname{rat}}(X,T)})) = (e(\alpha_f) y_f)_{f\in \mathcal{E}_{\operatorname{rat}}(X,T)}$. Note that $(Z,S)$ is a system.

    % Define
    
    % \begin{align*}
    %     \phi: X_{\operatorname{rat}} &\to \prod_{f\in \mathcal{E}_{\operatorname{rat}}(X,T)}\overline{f(X)}\\
    %         \pi(x) &\mapsto (f(x))_{f\in \mathcal{E}_{\operatorname{rat}}(X,T)}.
    % \end{align*}

    % Note that $\pi(x)=\pi(y)$ if and only if $\phi(\pi(x))=\phi(\pi(y))$. Therefore, $\phi$ is an isomorphism between $X_{\operatorname{rat}}$ and $\phi(X_{\operatorname{rat}})$. Consequently, it is enough to prove that $Z$ is equicontinuous. 
    
    % Let $(y_f)_{f\in \mathcal{E}_{\operatorname{rat}}(X,T)},(z_f)_{f\in \mathcal{E}_{\operatorname{rat}}(X,T)} \in Z$. For $f\in \mathcal{E}_{\operatorname{rat}}(X,T)$ and $n\in \Z$, we have\begin{align*}
    %     |e(n\alpha_f)y_f-e(n\alpha_f)z_f|= |y_f-z_f|.
    % \end{align*}

    % Hence, we deduce that $(Z,S)$ is equicontinuous.

    Now, suppose that $(X,T)$ is transitive. We will show that $(X_{\operatorname{rat}},T)$ is isomorphic to an adding machine. Let $x_0\in X$ be a transitive point, and set\begin{align*}
        \Lambda = \{\alpha\in\Q: Tf = e(\alpha)f \text{ for some nonzero }f\in \mathcal{E}_{\operatorname{rat}}(X,T)\} = \{\alpha_1,\alpha_2,\dots\}.
    \end{align*}

    Note that, for $f\in \mathcal{E}_{\operatorname{rat}}(X,T)$, we have $|Tf|=|f|$. Therefore, since $x_0$ is a transitive point and $f$ is continuous, $|f|$ is constant. In particular, $f(x)\neq 0$ for every $x\in X$.
    
    For $\alpha \in \Lambda$, choose  $f_\alpha\in \mathcal{E}_{\operatorname{rat}}(X,T)$ such that $f_\alpha(x_0)=1$ and $Tf_\alpha = e(\alpha)f_\alpha$. Since $x_0$ is a transitive point, we have $f_\alpha(X) = \langle e(\alpha)\rangle$. 

    Set $P = \prod_{\alpha\in \Lambda} \langle e(\alpha)\rangle$ and define\begin{align*}
        \psi: X&\to P \\ x &\mapsto (f_\alpha(x))_{\alpha\in \Lambda}.
    \end{align*}

    Let $f,g\in \mathcal{E}_{\operatorname{rat}}(X,T)$ such that $Tf = e(\alpha)f$ and $Tg = e(\alpha)g$ for some $\alpha\in\Q$. Since $g(x)\neq 0$ for all $x\in X$, it follows that $f/g$ is continuous and invariant. Consequently, $f/g$ is constant. Therefore, $f=cg$ for some $c\in \C$. Thus, every continuous eigenfunction associated with $e(\alpha)$ is a scalar multiple of $f_\alpha$. Hence, we deduce that $\psi(x)=\psi(y)$ if and only if $(x,y)\in \Krat(X,T)$.

    Since $x_0$ is a transitive point, we have that $\psi(X)= \overline{\{(e(n\alpha))_{\alpha\in \Lambda}: n\in\Z\}}$. Therefore, $X_{\operatorname{rat}}$ is isomorphic to $Y$, where $Y=\psi(X)$.

    For each $i\in \N$, let $r_i = \operatorname{lcm}(\operatorname{ord}(e(\alpha_1),\dots,\operatorname{ord}(e(\alpha_i)))$. Then $r_i\mid r_{i+1}$. Let $Y_i = \{(e(n\alpha_1),\dots,e(n\alpha_i)):n\in\Z\}$ be the projection of $Y$ onto its first $i$ coordinates. The map $\pi_i: \Z/r_i\Z\to Y_i$ given by $\pi_i(n)=(e(n\alpha_1),\dots,e(n\alpha_i))$ is a group isomorphism. Thus, we conclude that $Y$ is isomorphic to $\varprojlim \Z/r_i \Z$. Hence, $(X_{\operatorname{rat}},T)$ is isomorphic to an adding machine.
\end{proof}

We say that $(X,T)$ has \emph{continuous discrete rational spectrum} if $C(X)$ is the closed space spanned by all $f\in \mathcal{E}_{\operatorname{rat}}(X,T)$.

\begin{proposition}
    Let $(X,T)$ be a system. Then $X_{\operatorname{rat}}$ is the maximal factor of $(X,T)$ having continuous discrete rational spectrum.
\end{proposition}

\begin{proof}
    We first show that $X_{\operatorname{rat}}$ has continuous discrete rational spectrum. Let $\pi: X\to X_{\operatorname{rat}}$ be the quotient map. Note that, for each $f\in \mathcal{E}_{\operatorname{rat}}(X,T)$, there exists a $g_f\in C(X_{\operatorname{rat}})$ such that $f=g_f \circ \pi$. By the definition of $\Krat(X,T)$, if $\pi(x)\neq\pi(y)$, then there exists $f\in \mathcal{E}_{\operatorname{rat}}(X,T)$ such that $g_f(\pi(x))\neq g_f(\pi(y))$. Therefore, by Stone–Weierstrass, we conclude that $X_{\operatorname{rat}}$ has continuous discrete rational spectrum.

    Now, we show that $X_{\operatorname{rat}}$ is the maximal factor having continuous discrete rational spectrum. Let $\phi: X\to Y$ be a factor map such that $(Y,T)$ has continuous discrete rational spectrum.
    
    Let $(x,y)\in \Krat(X,T)$. For every $f\in \mathcal{E}_{\operatorname{rat}}(Y,T)$, we have $f \circ \phi \in \mathcal{E}_{\operatorname{rat}}(X,T)$. Therefore, $f(\phi(x))=f(\phi(y))$ for each $f\in \mathcal{E}_{\operatorname{rat}}(Y,T)$. Since $(Y,T)$ has continuous discrete rational spectrum, it follows that $f(\phi(x))=f(\phi(y))$ for each $f\in C(Y)$. Hence, we deduce that $\phi(x)=\phi(y)$. Therefore, we conclude that $Y$ is a factor of $X_{\operatorname{rat}}$.
\end{proof}

For a minimal system, the topological rational Kronecker factor introduced in \cite{Glasscock_Koutsogiannis_Richter_mult_comb_return:2019} is defined as the maximal factor that is an adding machine. Every adding machine has continuous discrete rational spectrum. On the other hand, by the preceding results, our factor $X_{\operatorname{rat}}$ is an adding machine and is maximal among the factors having continuous discrete rational spectrum. It follows that the two definitions of $X_{\operatorname{rat}}$ coincide for minimal systems.

For a minimal system $(X,T)$, the definition in \cite{Glasscock_Koutsogiannis_Richter_mult_comb_return:2019} implies that $X_{\operatorname{rat}}$ is trivial if and only if $(X,T)$ is totally minimal, meaning that $(X,T^n)$ is minimal for every $n\in\Z\setminus\{0\}$. It is therefore natural to ask whether the analogous statement holds for transitive systems: is $X_{\operatorname{rat}}$ trivial if and only if $T$ is totally transitive, meaning that $(X,T^n)$ is transitive for every $n\in\Z\setminus\{0\}$? The implication from triviality of $X_{\operatorname{rat}}$ to total transitivity is false. Indeed, let $X=\Z\cup\{\infty\}$ be the one-point compactification of $\Z$, and define $T$ by $T(n)=n+1$ for $n\in\Z$ and $T(\infty)=\infty$. The point $0$ is transitive. However, $T^2$ is not transitive, since every $T^2$-orbit contained in $\Z$ consists entirely of either even or odd integers. Let $f\in \mathcal{E}_{\operatorname{rat}}(X,T)$ and $\alpha\in\Q$ such that $Tf=e(\alpha)f$. Note that \begin{align*}
    e(\alpha) f(\infty)=Tf(\infty) = f(T(\infty))=f(\infty).
\end{align*}

As before, we have $f(x)\neq 0$ for all $x\in X$. Hence, we have $e(\alpha)=1$. Therefore, $f$ is invariant. By transitivity, we obtain that $f$ is constant. Thus, every $f\in \mathcal{E}_{\operatorname{rat}}(X,T)$ is constant, and hence $X_{\operatorname{rat}}$ is trivial.

Nevertheless, the converse implication remains valid: every totally transitive system has a trivial topological rational Kronecker factor.

\begin{lemma}
    Let $(X,T)$ be a system. If $T$ is totally transitive, then $\Krat(X,T)=X\times X$.
\end{lemma}

\begin{proof}
    Let $f\in \mathcal{E}_{\operatorname{rat}}(X,T)$. Then there exists $q\in \N$ such that $T^{q}f=f$. Since $T$ is totally transitive, there exists a dense G$_\delta$ subset $\Omega$ of $X$ such that \begin{align*}
        \overline{\{T^{qn}x:n\in\Z\}} =X
    \end{align*}

    for each $x\in \Omega$. Let $x\in \Omega$ and $y\in X$. There exists a sequence $(n_i)_{i\in\N}$ such that $T^{qn_i}x\to y$. Therefore,\begin{align*}
        f(x) = T^{qn_i} f(x) \to f(y).
    \end{align*}

    Consequently, $f(x)=f(y)$. Since $\Omega$ is dense and $f$ is continuous, we conclude that $f$ is constant. 
\end{proof}

Using the preceding lemma and assuming \cref{thm: TFC_distinct_deg}, we obtain the following result, which provides a partial answer to {\cite[Problem 8]{Koutsogiannis_Kuca_Sun_2_nil_avr_distinct_degree:2026}} for commuting transformations and distinct-degree polynomials.

\begin{corollary}
    Let $(X,S_1,\dots,S_k)$ be a minimal $\Z^k$-system, let $T_1,\dots,T_d\in \langle S_{1},\dots,S_{k}\rangle$ be transitive transformations, and $p_1,\dots,p_d$ be nonconstant integer polynomials with distinct degrees. If $T_i$ is totally transitive for each $i=1,\dots,d$, then there exists a dense G$_\delta$ subset $\Omega$ of $X$ such that \begin{align*}
        \{(T_1^{p_1(n)}x,\dots,T_d^{p_d(n)}x):n\in\Z\}
    \end{align*}

    is dense in $X^d$.
\end{corollary}

\subsection{Proof of \cref{thm: TFC_distinct_deg}}

Before proving \cref{thm: TFC_distinct_deg}, we state the following result. It will be proved as part of the proof of \cref{thm: TFC_distinct_deg} and used in the next section.

\begin{proposition}
    Let $\pi\colon (X,S_1,\dots,S_k) \to (Y,S_1,\dots,S_k)$ be an open factor between minimal $\Z^k$-systems, and let $T_1,\dots,T_d\in \langle S_{1},\dots,S_{k}\rangle$ be transitive transformations. If $Y$ is an almost one-to-one extension of $X_\infty$, then for any open subsets $V_{0},V_{1},\dots,V_{d}$ of $X$ with $\cap_{i=0}^{d}\pi(V_i)\neq \emptyset$, and for any $p_1,\dots,p_d$ nonconstant integer distinct-degree polynomials with $p_{i}(0)=0$ for $i=1,\dots, d$, there exists some $n\in \Z$ such that\begin{align*}
        V_{0} \cap T_{1}^{-p_1(n)}V_{1} \cap \dots\cap T_{d}^{-p_d(n)}V_{d} \neq \emptyset.
    \end{align*}
\end{proposition}

\cref{thm: TFC_distinct_deg} is a consequence of \cref{thm: TFC_nice_pol}. Part of its proof closely follows the proof of {\cite[Theorem A]{Ye_Yu_polynomial_saturation:2025}}. We therefore provide only a brief sketch of that part and refer to \cite{Ye_Yu_polynomial_saturation:2025} whenever the required argument follows by a straightforward modification of the corresponding argument given there.

\begin{proof}[Proof of \cref{thm: TFC_distinct_deg}]
     Let $\pi:X\to X_{\infty}$ be the quotient map. By the O-diagram, we may assume without loss of generality that $\pi$ is open.
     
     Let $V_{0},V_{1},\dots,V_{d}$ be open subsets of $X$ with $\cap_{i=0}^{d}\pi(V_i)\neq \emptyset$. Let $j_{1},\dots,j_{d}$ be distinct numbers in $\{1,\dots,d\}$ such that \begin{align*}
        \deg(p_{j_1})> \deg(p_{j_2}) > \dots >\deg(p_{j_{d}}).
    \end{align*}

    Let $\cP_{1} = (p_{j_1},0,0,\dots,0),\cP_{2} = (0,p_{j_2},0,\dots,0),\dots,\cP_{d} =(0,0,\dots,0,p_{j_d})$.  Notice that $(\cP_1,\dots,\cP_d)$ is a very nice family of polynomial $d$-tuples. By \cref{thm: TFC_nice_pol}, there exists $n\in \Z$ such that \begin{align*}
        V_{0} \cap T_{j_1}^{-p_{j_1}(n)} V_{j_1} \cap \dots\cap T_{j_d}^{-p_{j_d}(n)}V_{j_d} \neq \emptyset.
    \end{align*}

    By the proof of {\cite[Lemma 3.2]{Ye_Yu_polynomial_saturation:2025}}, there exists a dense G$_\delta$ subset $\Omega_1$ of $X$ such that for any $x\in \Omega_1$,\begin{align*}
        (\pi^{(d)})^{-1}(\pi^{(d)}(x^{(d)}))\subseteq \overline{\{ (T_1^{p_1(n)}x,\dots,T_{d}^{p_d(n)}x):n\in\Z\}}.
    \end{align*}

    Let $\mu$ be the unique ergodic invariant probability measure on $X_{\infty}$. Since the polynomials $p_1,\dots,p_d$ have distinct degrees, they are linearly independent. Hence, by {\cite[Theorem 2.10]{Frantzikinakis_Kuca_joint_erg_comm_poly:2025}}, for all $f_1,\dots,f_d\in L^{\infty}(\mu)$, we have\begin{align*}
        \lim_{N\to \infty}\left\|\dfrac{1}{N}\sum_{n=1}^{N} \prod_{i=1}^{d} T_{i}^{p_{i}(n)}f_i - \dfrac{1}{N}\sum_{n=1}^{N} \prod_{i=1}^{d} T_{i}^{p_{i}(n)}\mathbb{E}(f_i\mid \mathcal{K}_{\operatorname{rat}}(T_i)) \right\|_{L^2(\mu)}=0.
    \end{align*}

    Since $\RP^{[m]}(X,T_i)=\RP^{[m]}(X,T_j)$ for all $m\in\N$ and $i,j\in\{1,\dots,d\}$, we have $\Krat(X,T_i)=\Krat(X,T_j)$. Let $\pi_{\infty}: X_{\infty} \to X_{\operatorname{rat}} = X/\Krat(X,T_1)$ be the factor map. Since $X_{\infty}$ is a pro-nilsystem, we have $X_{\operatorname{rat}}=Z_{\operatorname{rat}}$. Therefore, by the same proof as {\cite[Theorem 3.7]{Ye_Yu_polynomial_saturation:2025}}, we obtain a dense G$_\delta$ subset $\Omega_2$ of $X_\infty$ such that for any $x\in \Omega_2$,\begin{align*}
        (\pi_{\infty}^{(d)})^{-1}(\pi_{\infty}^{(d)}(x^{(d)}))\subseteq \overline{\{ (T_1^{p_1(n)}x,\dots,T_{d}^{p_d(n)}x):n\in\Z\}}.
    \end{align*}

    Let $\Omega_{3}$ be the dense G$_\delta$ subset of $X_\infty$ consisting of the continuity points of the map $y\mapsto \pi^{-1}(y)$. Moreover, there exists a dense G$_\delta$ subset $\Omega_4$ of $X_{\operatorname{rat}}$ such that, for each $y\in \Omega_4$, $\pi_{\infty}^{-1}(y)\cap \Omega_{1} \cap \pi^{-1}(\Omega_2 \cap \Omega_3)$ is a dense G$_\delta$ subset of $\pi_{\infty}^{-1}(y)$. The existence of $\Omega_3$ and $\Omega_4$ follows from {\cite[Lemma 2.7]{Ye_Yu_polynomial_saturation:2025}} and {\cite[Proposition 3.1]{Veech_point-distal_flows:1970}} respectively. Finally, following the proof of {\cite[Theorem A]{Ye_Yu_polynomial_saturation:2025}}, we conclude that $\Omega = \pi_{\operatorname{rat}}^{-1}(\Omega_4)\cap \Omega_{1} \cap \pi^{-1}(\Omega_2 \cap \Omega_3)$ is the required set, where $\pi_{\operatorname{rat}}: X\to X_{\operatorname{rat}}$ is the quotient map.
\end{proof}

\section{Joint transitivity}

In this section, we prove \cref{thm:B}. First, we state a characterization of joint transitivity that will be used.

\begin{lemma}[{\cite[Lemma 2.1]{Donoso_Koutsogiannis_Sun_joint_transitivity:2025}}]
    Let $(X,S_1,\dots,S_k)$ be a minimal $\Z^k$-system and $(a_1(n))_n,\dots,$ $(a_d(n))_n$ be sequences with values in $\Z^k$. The following are equivalent:\begin{enumerate}
        \item There exists a dense G$_\delta$ subset $\Omega$ of $X$ such that the set\begin{align*}
            \{(S_{a_1(n)}x,\dots,S_{a_d(n)}x)\colon n\in\Z\}
        \end{align*}

        is dense in $X^d$ for every $x\in \Omega$.

        \item For every collection of nonempty open subsets $V_0,V_1,\dots,V_d$ of $X$, there is some $n\in \Z$ such that\begin{align*}
            V_0 \cap S_{-a_1(n)}V_1 \cap \dots \cap S_{-a_d(n)}V_d \neq \emptyset.
        \end{align*}
    \end{enumerate}
\end{lemma}

We first prove the nontrivial implication in the case where the system is, up to an almost one-to-one extension, a pro-nilsystem. For this, we use the following characterization of joint ergodicity.

\begin{theorem}[{\cite{Frantzikinakis_Kuca_joint_erg_comm_poly:2025}}]\label{thm: joint_ergodic}
    Let $(X,\mu,T_1,\dots,T_d)$ be a measure preserving system such that $T_1,\dots,T_d$ are invertible, commuting, and ergodic transformations, and let $p_1,\dots,p_d$ be nonconstant integer polynomials with distinct degrees and $p_i(0)=0$. Then $(T_1^{p_1(n)})_{n},\dots,(T_d^{p_d(n)})_{n}$ are jointly ergodic if and only if $(T_1^{p_1(n)}\times \dots\times T_{d}^{p_{d}(n)})_{n}$ is ergodic for $\mu\times\dots\times\mu$.
\end{theorem}

\begin{lemma}\label{lemma: nil_case}
    Let $(X,S_1,\dots,S_k)$ be a minimal $\Z^k$-system which is an almost one-to-one extension of a pro-nilsystem, let $T_1,\dots,T_d\in\langle S_1,\dots,S_k \rangle$ be transitive transformations, and let $p_1,\dots,p_d$ be nonconstant integer polynomials with distinct degrees and $p_i(0)=0$ for $i=1,\dots,d$.  If $(T_1^{p_1(n)}\times \dots\times T_{d}^{p_{d}(n)})_{n}$ is transitive on $X^d$, then $(T_1^{p_1(n)})_{n},\dots,(T_d^{p_d(n)})_{n}$ are jointly transitive.
\end{lemma}

\begin{proof}
    It is enough to prove the result in the case where $X$ is a pro-nilsystem. Thus, we may assume from now on that $X$ is a pro-nilsystem. Let $V_0,V_1,\dots,V_d$ be nonempty open subsets of $X$.

    Since $X$ is a pro-nilsystem and the transformations $T_1,\dots,T_d$ and the sequence $(T_1^{p_1(n)}\times \dots\times T_{d}^{p_{d}(n)})_{n}$ are transitive, it follows that $T_1,\dots,T_d$ are ergodic and $(T_1^{p_1(n)}\times \dots\times T_{d}^{p_{d}(n)})_{n}$ is ergodic for $\mu\times\dots\times\mu$, where $\mu$ is the Haar measure on $X$. Therefore, by \cref{thm: joint_ergodic}, $(T_1^{p_1(n)})_{n},\dots,(T_d^{p_d(n)})_{n}$ are jointly ergodic. Hence,\begin{align*}
        \lim_{N\to\infty}\dfrac{1}{N}\sum_{n=1}^{N} \mu(V_0 \cap T_{1}^{-p_{1}(n)}V_1 \cap\dots\cap T_{d}^{-p_d(n)}V_d) = \mu(V_0)\mu(V_1)\dots\mu(V_d). 
        \end{align*}

        Since $(X,S_1,\dots,S_k)$ is minimal, every nonempty open subset of $X$ has positive Haar measure. Thus the limit above is positive, and consequently there exists $n\in \N$ such that\begin{align*}
            V_0 \cap T_{1}^{-p_{1}(n)}V_1 \cap\dots\cap T_{d}^{-p_d(n)}V_d\neq \emptyset.
        \end{align*} 
\end{proof}

To pass from this case to an arbitrary minimal system, we need the following lemma.

\begin{lemma}\label{lemma: lift_transitive_almost_onetoone}
    Let $\pi\colon (X,S_1,\dots,S_k)\to (Y,S_1,\dots,S_k)$ be an almost one-to-one factor between minimal $\Z^k$-systems. Let $T_1,\dots,T_d\in \langle S_1,\dots,S_k\rangle$ and let $p_1,\dots,p_d\in \Z[t]$. Then $(T_1^{p_1(n)}\times \dots\times T_{d}^{p_{d}(n)})_{n}$ is transitive on $X^d$ if and only if $(T_1^{p_1(n)}\times \dots\times T_{d}^{p_{d}(n)})_{n}$ is transitive on $Y^d$.
\end{lemma}

\begin{proof}
    The forward implication is clear. Conversely, assume that $(T_1^{p_1(n)}\times \dots\times T_{d}^{p_{d}(n)})_{n}$ is transitive on $Y^d$. Let $\Omega$ be the dense G$_\delta$ subset of $X^d$ such that $(\pi^{(d)})^{-1}(\pi^{(d)}(\bx))=\{\bx\}$ for every $\bx\in\Omega$. Let $\by\in Y^d$ be a transitive point for the sequence $(T_1^{p_1(n)}\times \dots\times T_{d}^{p_{d}(n)})_{n}$ on $Y^d$, and choose $\bx\in X^d$ such that $\pi^{(d)}(\bx)=\by$. We claim that $\{T_1^{p_1(n)}\times \dots\times T_{d}^{p_{d}(n)} \bx:n\in \Z\}$ is dense in $X^d$. 

    Let $\bx'\in \Omega$. Since $\by$ is transitive on $Y^d$ for the sequence $(T_1^{p_1(n)}\times \dots\times T_{d}^{p_{d}(n)})_{n}$, there exists a sequence $(n_i)_i\subseteq \Z$ such that \begin{align*}
        T_1^{p_1(n_i)}\times \dots\times T_{d}^{p_{d}(n_i)} \pi^{(d)}(\bx) \to \pi^{(d)}(\bx').
    \end{align*}

    We may assume that $T_1^{p_1(n_i)}\times \dots\times T_{d}^{p_{d}(n_i)}\bx\to \bz$ for some $\bz\in X^d$. Since $\bx'\in\Omega$, it follows that $\bz=\bx' \in \overline{\{T_1^{p_1(n)}\times\dots\times T_{d}^{p_{d}(n)} \bx:n\in \Z\}}$. Since $\Omega$ is dense, we conclude that $\{T_1^{p_1(n)}\times \dots\times T_{d}^{p_{d}(n)} \bx:n\in \Z\}$ is dense in $X^d$.
\end{proof}

\begin{proof}[Proof of \cref{thm:B}]
    Assume first that $(T_1^{p_1(n)})_{n},\dots,(T_d^{p_d(n)})_{n}$ are jointly transitive. Then there exists $x\in X$ such that $\{(T_1^{p_1(n)}x,\dots,T_d^{p_d(n)}x):n\in \Z\}$ is dense in $X^d$. Hence $(T_1^{p_1(n)}\times \dots\times T_{d}^{p_{d}(n)})_{n}$ is transitive on $X^d$.
    
    Conversely, assume that $(T_1^{p_1(n)}\times \dots\times T_{d}^{p_{d}(n)})_{n}$ is transitive on  $X^d$. By the O-diagram, we may consider almost one-to-one extensions $X^*$, $X_\infty^*$ of $X$ and $X_\infty$ respectively for which the factor $\pi^*: X^{*}\to X_\infty^*$ is open. By \cref{lemma: lift_transitive_almost_onetoone}, $(T_1^{p_1(n)}\times \dots\times T_{d}^{p_{d}(n)})_{n}$ is transitive on the product space $(X^*)^d$. Then $(T_1^{p_1(n)}\times \dots\times T_{d}^{p_{d}(n)})_{n}$ is transitive on both $X_\infty^d$ and $(X_\infty^*)^d$.

    Let $V_0,V_1,\dots,V_d$ be nonempty open subsets of $X^*$. Since $\pi^*$ is open, the sets $\pi^*(V_0),\pi^*(V_1),\dots,\pi^*(V_d)$ are nonempty open subsets of $X_\infty^*$. By \cref{lemma: nil_case}, there exists $n\in \Z$ such that \begin{align*}
        \pi^{*}(V_0)\cap T_1^{-p_1(n)}\pi^*(V_1) \cap\dots\cap T_d^{-p_d(n)}\pi^*(V_d)\neq \emptyset.
    \end{align*}

    Applying \cref{thm: TFC_distinct_deg} to $\partial_n p_1,\dots,\partial_n p_d$ and the open sets $V_0, T_1^{-p_1(n)}V_1, \dots, T_d^{-p_d(n)}V_d$, we obtain $k\in \Z$ such that\begin{align*}
            V_0\cap T_1^{-p_1(n+k)}V_1\cap \dots\cap T_d^{-p_d(n+k)}V_d \neq \emptyset.
    \end{align*}

    Therefore $(T_1^{p_1(n)})_n,\dots,(T_d^{p_d(n)})_n$ are jointly transitive.
\end{proof}

\bibliographystyle{abbrv}

\begin{thebibliography}{10}

\bibitem{Alvarez_regionally_proximal_tfc_group_actions:2026}
A.~{\'A}lvarez.
\newblock On higher order regionally proximal relations and topological characteristic factors for group actions.
\newblock Preprint, {arXiv}:2605.02304 [math.{DS}] (2026), 2026.

\bibitem{Alvarez_Donoso_cube_struct_univ_nil_applications:2025}
A.~{\'A}lvarez and S.~Donoso.
\newblock Cube structures of the universal minimal system, nilsystems and applications.
\newblock To appear in \emph{Transactions of the American Mathematical Society}.

\bibitem{Auslander_minimal_flows_and_extensions:1988}
J.~Auslander.
\newblock {\em Minimal flows and their extensions}, volume 153 of {\em North-Holland Mathematics Studies}.
\newblock North-Holland Publishing Co., Amsterdam, 1988.
\newblock Notas de Matem\'{a}tica, 122. [Mathematical Notes].

\bibitem{Berend_Bergelson_joint_ergodicity:1984}
D.~Berend and V.~Bergelson.
\newblock Jointly ergodic measure-preserving transformations.
\newblock {\em Israel J. Math.}, 49(4):307--314, 1984.

\bibitem{Bergelson_WM_PET:1987}
V.~Bergelson.
\newblock Weakly mixing {PET}.
\newblock {\em Ergodic Theory Dynam. Systems}, 7(3):337--349, 1987.

\bibitem{Bergelson_Leibman96}
V.~Bergelson and A.~Leibman.
\newblock Polynomial extensions of van der {W}aerden's and {S}zemer\'edi's theorems.
\newblock {\em J. Amer. Math. Soc.}, 9(3):725--753, 1996.

\bibitem{Bergelson_Leibman_Son_joint_erg_generalized_linear:2016}
V.~Bergelson, A.~Leibman, and Y.~Son.
\newblock Joint ergodicity along generalized linear functions.
\newblock {\em Ergodic Theory Dynam. Systems}, 36(7):2044--2075, 2016.

\bibitem{Bressaud_Durand_Maass_Eigenvalues_finite_rank:2010}
X.~Bressaud, F.~Durand, and A.~Maass.
\newblock On the eigenvalues of finite rank {B}ratteli-{V}ershik dynamical systems.
\newblock {\em Ergodic Theory Dynam. Systems}, 30(3):639--664, 2010.

\bibitem{Chu_Frantzikinakis_Host_ergodic_averages_distinct_degree:2011}
Q.~Chu, N.~Frantzikinakis, and B.~Host.
\newblock Ergodic averages of commuting transformations with distinct degree polynomial iterates.
\newblock {\em Proc. Lond. Math. Soc. (3)}, 102(5):801--842, 2011.

\bibitem{deVries_elements_topological_dynamics:1993}
J.~de~Vries.
\newblock {\em Elements of topological dynamics}, volume 257 of {\em Mathematics and its Applications}.
\newblock Kluwer Academic Publishers Group, Dordrecht, 1993.

\bibitem{Dong_Donoso_Maass_Shao_Ye_infinite_step_nil:2013}
P.~Dong, S.~Donoso, A.~Maass, S.~Shao, and X.~Ye.
\newblock Infinite-step nilsystems, independence and complexity.
\newblock {\em Ergodic Theory Dynam. Systems}, 33(1):118--143, 2013.

\bibitem{Donoso_Ferre_Koutsogiannis_Sun_multicorr_joint_erg:2024}
S.~Donoso, A.~Ferr\'{e}~Moragues, A.~Koutsogiannis, and W.~Sun.
\newblock Decomposition of multicorrelation sequences and joint ergodicity.
\newblock {\em Ergodic Theory Dynam. Systems}, 44(2):432--480, 2024.

\bibitem{Donoso_Koutsogiannis_Kuca_Sun_Tsinas_resolving_joint_ergodicity_Hardy:2025}
S.~Donoso, A.~Koutsogiannis, B.~Kuca, W.~Sun, and K.~Tsinas.
\newblock Resolving the joint ergodicity problem for {Hardy} sequences.
\newblock Preprint, {arXiv}:2506.20459 [math.{DS}] (2025), 2025.

\bibitem{Donoso_Koutsogiannis_Kuca_Sun_Tsinas_seminorm_joint_ergodicity_independent_Hardy:2025}
S.~Donoso, A.~Koutsogiannis, B.~Kuca, W.~Sun, and K.~Tsinas.
\newblock Seminorm estimates and joint ergodicity for pairwise independent {Hardy} sequences.
\newblock Preprint, {arXiv}:2410.15130 [math.{DS}] (2025), 2025.

\bibitem{Donoso_Koutsogiannis_Sun_seminorms_polynomials_joint_ergodicity:2022}
S.~Donoso, A.~Koutsogiannis, and W.~Sun.
\newblock Seminorms for multiple averages along polynomials and applications to joint ergodicity.
\newblock {\em J. Anal. Math.}, 146(1):1--64, 2022.

\bibitem{Donoso_Koutsogiannis_Sun_joint_erg_poly_growth:2023}
S.~Donoso, A.~Koutsogiannis, and W.~Sun.
\newblock Joint ergodicity for functions of polynomial growth.
\newblock {\em Israel J. Math.}, 268(1):315--363, 2025.

\bibitem{Donoso_Koutsogiannis_Sun_joint_transitivity:2025}
S.~Donoso, A.~Koutsogiannis, and W.~Sun.
\newblock Joint transitivity for linear iterates.
\newblock {\em Forum Math. Sigma}, 13:Paper No. e34, 13, 2025.

\bibitem{Frantzikinakis_joint_erg_primes:2022}
N.~Frantzikinakis.
\newblock Joint ergodicity of fractional powers of primes.
\newblock {\em Forum Math. Sigma}, 10:Paper No. e30, 30, 2022.

\bibitem{Frantzikinakis_joint_ergodicity_sequences:2023}
N.~Frantzikinakis.
\newblock Joint ergodicity of sequences.
\newblock {\em Adv. Math.}, 417:Paper No. 108918, 63, 2023.

\bibitem{Frantzikinakis_Kuca_seminorm_control_pairwise_dependent_pol:2023}
N.~Frantzikinakis and B.~Kuca.
\newblock Seminorm control for ergodic averages with commuting transformations along pairwise dependent polynomials.
\newblock {\em Ergodic Theory Dynam. Systems}, 43(12):4074--4137, 2023.

\bibitem{Frantzikinakis_Kuca_joint_erg_comm_poly:2025}
N.~Frantzikinakis and B.~Kuca.
\newblock Joint ergodicity for commuting transformations and applications to polynomial sequences.
\newblock {\em Invent. Math.}, 239(2):621--706, 2025.

\bibitem{Furstenberg_ergodic_szemeredi:1977}
H.~Furstenberg.
\newblock Ergodic behavior of diagonal measures and a theorem of {S}zemer\'{e}di on arithmetic progressions.
\newblock {\em J. Anal. Math.}, 31:204--256, 1977.

\bibitem{Furstenberg_Katznelson85}
H.~Furstenberg and Y.~Katznelson.
\newblock An ergodic {S}zemer\'edi theorem for {IP}-systems and combinatorial theory.
\newblock {\em J. d'Analyse Math.}, 45:117--168, 1985.

\bibitem{Furstenberg_Weiss_ergodic_thm_double:1996}
H.~Furstenberg and B.~Weiss.
\newblock A mean ergodic theorem for {$(1/N)\sum^N_{n=1}f(T^nx)g(T^{n^2}x)$}.
\newblock 5:193--227, 1996.

\bibitem{Glasner_top_erg_decomposition:1994}
E.~Glasner.
\newblock Topological ergodic decompositions and applications to products of powers of a minimal transformation.
\newblock {\em J. Anal. Math.}, 64:241--262, 1994.

\bibitem{Glasner_ergodic_theory_joinings:2003}
E.~Glasner.
\newblock {\em Ergodic theory via joinings}, volume 101 of {\em Mathematical Surveys and Monographs}.
\newblock American Mathematical Society, Providence, RI, 2003.

\bibitem{Glasner_Gutman_Ye_higher_regionallyproximal_general_groups:2018}
E.~Glasner, Y.~Gutman, and X.~Ye.
\newblock Higher order regionally proximal equivalence relations for general minimal group actions.
\newblock {\em Adv. Math.}, 333:1004--1041, 2018.

\bibitem{Glasner_Huang_Shao_Weiss_Ye_Topological_characteristic_factors:2020}
E.~Glasner, W.~Huang, S.~Shao, B.~Weiss, and X.~Ye.
\newblock Topological characteristic factors and nilsystems.
\newblock {\em J. Eur. Math. Soc. (JEMS)}, 27(1):279--331, 2025.

\bibitem{Glasscock_Koutsogiannis_Richter_mult_comb_return:2019}
D.~Glasscock, A.~Koutsogiannis, and F.~K. Richter.
\newblock Multiplicative combinatorial properties of return time sets in minimal dynamical systems.
\newblock {\em Discrete Contin. Dyn. Syst.}, 39(10):5891--5921, 2019.

\bibitem{Gutman_Manners_Varju_nilspaces_III:2020}
Y.~Gutman, F.~Manners, and P.~P. Varj{\'u}.
\newblock The structure theory of nilspaces. {III}: {Inverse} limit representations and topological dynamics.
\newblock {\em Adv. Math.}, 365:53, 2020.

\bibitem{Host_Kra_nonconventional_averages_nilmanifolds:2005}
B.~Host and B.~Kra.
\newblock Nonconventional ergodic averages and nilmanifolds.
\newblock {\em Ann. of Math. (2)}, 161(1):397--488, 2005.

\bibitem{Host_Kra_nilpotent_structures_ergodic_theory:2018}
B.~Host and B.~Kra.
\newblock {\em Nilpotent structures in ergodic theory}, volume 236 of {\em Mathematical Surveys and Monographs}.
\newblock American Mathematical Society, Providence, RI, 2018.

\bibitem{Host_Kra_Maass_nilstructure:2010}
B.~Host, B.~Kra, and A.~Maass.
\newblock Nilsequences and a structure theorem for topological dynamical systems.
\newblock {\em Adv. Math.}, 224(1):103--129, 2010.

\bibitem{Huang_Shao_Ye_nilbohr_automorphy:2016}
W.~Huang, S.~Shao, and X.~Ye.
\newblock Nil {B}ohr-sets and almost automorphy of higher order.
\newblock {\em Mem. Amer. Math. Soc.}, 241(1143):v+83, 2016.

\bibitem{Huang_Shao_Ye_top_correspondence_multiple_averages:2019}
W.~Huang, S.~Shao, and X.~Ye.
\newblock Topological correspondence of multiple ergodic averages of nilpotent group actions.
\newblock {\em J. Anal. Math.}, 138(2):687--715, 2019.

\bibitem{Koutsogiannis_Kuca_Sun_2_nil_avr_distinct_degree:2026}
A.~Koutsogiannis, B.~Kuca, and W.~Sun.
\newblock Structure of 2-step nilpotent ergodic averages for distinct-degree polynomials.
\newblock Preprint, {arXiv}:2607.29368 [math.{DS}] (2026), 2026.

\bibitem{Koutsogiannis_Sun_total_joint_ergodicity:2023}
A.~Koutsogiannis and W.~Sun.
\newblock Total joint ergodicity for totally ergodic systems.
\newblock 2023.
\newblock Preprint. arXiv:2302.12278.

\bibitem{Leibman05a}
A.~Leibman.
\newblock Pointwise convergence of ergodic averages for polynomial sequences of translations on a nilmanifold.
\newblock {\em Ergodic Theory Dynam. Systems}, 25(1):201--213, 2005.

\bibitem{Qiu_poly_orbits_tot_minimal:2023}
J.~Qiu.
\newblock Polynomial orbits in totally minimal systems.
\newblock {\em Adv. Math.}, 432:Paper No. 109260, 34, 2023.

\bibitem{Qiu_Xu_Ye_Yu_saturation_product:2025}
J.~Qiu, H.~Xu, X.~Ye, and J.~Yu.
\newblock Saturation of product systems and applications.
\newblock {\em Proc. Steklov Inst. Math.}, 330(1):342--358, 2025.

\bibitem{Qiu_Yu_saturated_cubes_measure:2023}
J.~Qiu and J.~Yu.
\newblock Saturated theorem along cubes for a measure and applications.
\newblock Preprint, {arXiv}:2311.14198 [math.{DS}] (2023), 2023.

\bibitem{Shao_Xu_saturation_R_flows:2025}
S.~Shao and H.~Xu.
\newblock Structure theorems of commuting transformations and minimal {$\Bbb{R} $}-flows.
\newblock {\em Discrete Contin. Dyn. Syst.}, 52:70--114, 2026.

\bibitem{Shao_Ye_regionally_prox_orderd:2012}
S.~Shao and X.~Ye.
\newblock Regionally proximal relation of order {$d$} is an equivalence one for minimal systems and a combinatorial consequence.
\newblock {\em Adv. Math.}, 231(3-4):1786--1817, 2012.

\bibitem{Tsinas_joint_erg_Hardy:2023}
K.~Tsinas.
\newblock Joint ergodicity of {H}ardy field sequences.
\newblock {\em Trans. Amer. Math. Soc.}, 376(5):3191--3263, 2023.

\bibitem{Veech_point-distal_flows:1970}
W.~A. Veech.
\newblock Point-distal flows.
\newblock {\em Amer. J. Math.}, 92:205--242, 1970.

\bibitem{Walsh12}
M.~N. Walsh.
\newblock Norm convergence of nilpotent ergodic averages.
\newblock {\em Ann. of Math. (2)}, 175(3):1667--1688, 2012.

\bibitem{Wu_Yu_saturation_product_pol:2026}
Q.~Wu and J.~Yu.
\newblock Saturation for product systems of polynomials.
\newblock Preprint, {arXiv}:2605.24529 [math.{DS}] (2026), 2026.

\bibitem{Ye_Yu_polynomial_saturation:2025}
X.~Ye and J.~Yu.
\newblock A refined saturation theorem for polynomials and applications.
\newblock {\em Proc. Amer. Math. Soc.}, 153(3):1077--1092, 2025.

\bibitem{Zhang_Zhao_topological_mult_rec_WM_GP:2021}
R.~F. Zhang and J.~J. Zhao.
\newblock Topological multiple recurrence of weakly mixing minimal systems for generalized polynomials.
\newblock {\em Acta Math. Sin. (Engl. Ser.)}, 37(12):1847--1874, 2021.

\end{thebibliography}

\end{document}